\documentclass[11pt]{amsart}

\usepackage{amsmath,amsfonts,eucal}
\newtheorem{thm}{Theorem}[section]
\newtheorem{prop}[thm]{Proposition}
\newtheorem{cor}[thm]{Corollary}
\newtheorem{lem}[thm]{Lemma}

\usepackage[top=1.4in, bottom=1.4in, left=1.7in, right=1.7in]{geometry}

\begin{document}

\def\C{{\mathbb C}}
\def\Z{{Z}}
\def\W{W}
\def\D{{\mathcal D}}
\def\G{{\mathcal G}}
\def\F{{\mathcal F}}
\def\H{{h}}
\def\X{{\mathcal X}}
\def\U{{u}}
\def\S{{a}}

\def\x{{\mathbf x}}
\def\y{{\mathbf y}}
\def\z{{\mathbf z}}
\def\ze{{\mathbf w}}

\def\1{{\mathbf 1}}

\def\bZ{{\mathbb Z}}
\def\N{{\mathbb N}}
\def\R{{\mathbb R}}
\def\E{{\mathbb E}}

\def\sgn{\mbox{sgn}}

\def\l{\lambda}
\def\w{\diamond}

\title[A multi-layer extension of the stochastic heat equation]{A multi-layer extension of the\\ stochastic heat equation}

 \author{Neil O'Connell}
 \address{Mathematics Institute, University of Warwick, Coventry CV4 7AL, UK}
\curraddr{}
\email{n.m.o-connell@warwick.ac.uk}
\thanks{}
 \author{Jon Warren}
 \address{Department of Statistics, University of Warwick, Coventry CV4 7AL, UK}
\curraddr{}
\email{j.warren@warwick.ac.uk}
\thanks{}



\date{}

\dedicatory{}

\maketitle

\begin{abstract}
Motivated by recent developments on solvable directed polymer models,
we define a `multi-layer' extension of the stochastic heat equation involving
non-intersecting Brownian motions. 
\end{abstract}

\section{Introduction and summary}

We consider the stochastic heat equation in one dimension
\begin{equation}\label{she1}
d u=\frac12 \partial_y^2 u \, dt + u\, dW(t) ,
\end{equation}
with initial condition $u(0,x,y)=\delta(x-y)$, where $W$ is a cylindrical Brownian 
motion on $L^2(\R)$ and the term $u\, dW(t)$ is interpreted as an It\^o integral.
Formally, we can write this as
\begin{equation}\label{she2}
\partial_t u=\frac12 \partial_y^2 u +  u  \dot{W} ,
\end{equation}
where $\dot{W}$ denotes space-time white noise associated with  the Brownian motion $W$. 
This is a distribution-valued
Gaussian field on $[0,\infty)\times\R$ with covariance function
$$\E\,\dot{W}(t,y)\dot{W}(t',y')=\delta(y-y')\delta(t-t').$$
For more background see, for example,~\cite{bc,houz,q,walsh}.

We will write stochastic integrals $\int_0^t H(s) dW(s)$ of a predictable $L^2(\R)$-valued process $H$ as
\[
\int_0^t H(s,x) W(ds,dx).
\]
Then the solution $u(t,x,y)$ to \eqref{she1} is given by the chaos expansion 
\begin{align}\label{chaos}
u(t,x,y)=p(t,x,y)+& \nonumber \\
\sum_{k=1}^\infty \int_{\Delta_k(t)} \int_{\R^k}
p(t_1,x,x_1) & p(t_2-t_1,x_1,x_2)\cdots p(t-t_k,x_k,y)\\
&\times \W(dt_1,dx_1)\cdots \W(dt_k,dx_k),\nonumber
\end{align} 
where $\Delta_k(t)=\{0<t_1<\cdots<t_k<t\}$ and
$$p(t,x,y)=\frac1{\sqrt{2\pi t}} e^{-(x-y)^2/2t}.$$ 
For each $t>0$ and $x,y\in\R$, the expansion (\ref{chaos}) is convergent in $L^2(W)$.
It satisfies (\ref{she1}) in the sense that it satisfies the integral equation
\begin{equation}\label{ie-she1}
u(t,x,y)=p(t,x,y)+  \int_0^t \int_\R p(t-s,y',y) u(s,x,y') \W(ds, dy' ).
\end{equation} 
For more details see, for example,~\cite{q}.  The chaos series representation for the solution can 
be viewed as an expansion of the Feynman-Kac formula.  To emphasize this interpretation,
the solution (\ref{chaos}) is often written as a generalised Wiener functional
\begin{equation}\label{gwf}
u(t,x,y) =  p(t,x,y) \E \exp^\diamond \left( \int_0^t dW(s,X_s) \right) ,
\end {equation}
 where the expectation is with respect to a Brownian bridge $(X_s,\ s\le t)$ which starts at $x$ at time $0$ and ends at $y$ at time $t$~\cite{houz}.  The factor $p(t,x,y)$ is arises because of this use of the bridge in the representation, corresponding to the $\delta$-function initial condition,   rather than the more usual free Brownian motion. It is important to note that the integral 
 $\int_0^t dW(s,X_s) ds$, representing the integral of the white noise $\dot{W}$ along the path $s \mapsto (s, X_s)$, is not well-defined, and equation \eqref{gwf} is purely formal. 

This solution arises as a scaling limit of partition functions associated with
lattice directed polymers, in the `intermediate disorder' regime~\cite{akq,mq}.  
In fact, as suggested by (\ref{gwf}), it can be interpreted directly as a partition function for 
the continuum random polymer~\cite{akq1}.  
On the other hand, it has been known for some time that $h=\log u$ arises as the scaling
limit of the height profile of the weakly asymmetric simple exclusion process, at least for
equilibrium initial conditions~\cite{bg}, recently extended to include the present initial 
condition in \cite{acq}.
With this `surface growth' interpretation, $h$ is understood to be the
physically relevant solution (also known as the Cole-Hopf solution) to the 
KPZ (Kardar-Parisi-Zhang) equation~\cite{kpz}
$$\partial_t h = \frac12  \partial_y^2 h +\frac12  (\partial_y h)^2 + \dot{W}(t,y),$$
with `narrow wedge' initial condition.

In a remarkable recent development
the exact distribution of the random variable $u(t,x,y)$ has been determined. 
This has been acheived via two distinct approaches.  One of these~\cite{acq,ss1,ss2,ss3,ss4} uses the asymmetric 
simple exclusion process approximation together with recent work by Tracy and Widom~\cite{tw1,tw2,tw3,tw4} in 
which exact formulas have been obtained for that process using an approach based on the Bethe ansatz.  
Another approach~\cite{cdr,dot,dot2,dk} is based on replicas, where the moments of $u(t,x,y)$ are 
related to the attractive $\delta$-Bose gas and computed via the Bethe ansatz.  
These developments indicate that there is an underlying integrable structure behind the KPZ and 
stochastic heat equations which is not yet fully understood.

On the other hand, it has recently been found that there are exactly solvable discrete (or semi-discrete) directed 
polymer models~\cite{cosz,mo,noc,oy,osz,s,sv}, yielding yet another approach.  For these models, there is a direct 
connection to integrable systems (specifically to the quantum Toda lattice) 
which one might hope to understand in the continuum scaling limit.
We will describe here one of the main results of the paper~\cite{noc}, which provided the motivation 
for the present work. 

Define an `up/right path' in $\R\times\bZ$ to be an increasing path which either proceeds to the right
or jumps up by one unit.  For each sequence $0<t_1<\cdots<t_{N-1}<t$ we can 
associate an up/right path $\pi$ from $(0,1)$ to $(t,N)$ which has jumps between the 
points $(t_i,i)$ and $(t_i,i+1)$, for $i=1,\ldots,N-1$, and is continuous otherwise.  
Let $B(t)=(B_1(t),\ldots,B_N(t))$, $t\ge 0$, be a standard Brownian motion in $\R^N$ 
and define
$$Z^N(t) = \int e^{E(\pi)} d\pi ,$$
where
$$E(\pi) = B_1(t_1)+B_2(t_2)-B_2(t_1)+\cdots + B_N(t)-B_N(t_{N-1})$$
and the integral is with respect to Lebesgue measure on the Euclidean set
\begin{equation}
\label{paths} \{(t_1,\ldots,t_{N-1})\in\R^{N-1}:\ 0<t_1<\cdots<t_{N-1}<t\}\end{equation}
of all such paths.  This is the partition function for the model, which was introduced in \cite{oy}.  
In~\cite{noc} an explicit integral formula is obtained for the Laplace transform of the distribution 
of $Z^N(t)$, via the following `multi-layer' construction.  For $n=1,2,\ldots,N$, define
\begin{equation}\label{sdpfs}
Z_n^N(t) = \int e^{E(\pi_1)+\cdots+E(\pi_n)} d\pi_1\ldots d\pi_n ,
\end{equation}
where the integral is with respect to Lebesgue measure on the set of $n$-tuples
of non-intersecting (disjoint) up/right paths with respective initial points $(0,1),\ldots,(0,n)$ 
and respective end points $(t,N-n+1),\ldots,(t,N)$.  Here, the notion of Lebesgue measure
arises by identifying this set of paths as a suitable subset of $\R^{n(N-n)}$, as in \eqref{paths}.
Define $X_1^N(t)=\log Z_1^N(t)$ and, for $n\ge 2$, 
$X_n^N(t)=\log [Z^N_n(t)/Z^N_{n-1}(t)]$.
The relevance of this construction is analogous to the role of the RSK correspondence 
in the study of last passage percolation and longest increasing subsequence problems;
in this setting it is based on a geometric variant of the RSK correspondence.
The main result in~\cite{noc} is the following.
\begin{thm}\label{qt} 
The process $X^N(t)=(X^N_1(t),\ldots,X^N_N(t))$, $t>0$, is a diffusion process in 
$\R^N$ with infinitesimal generator given by
\begin{equation}\label{pd}
\mathcal{L}=\frac12 \Delta + \nabla \log \psi_0\cdot\nabla
\end{equation}
where $\psi_0$ is a (particular) ground state eigenfunction of the quantum Toda lattice
Hamiltonian
$$H=-\Delta+2 \sum_{i=1}^{N-1} e^{x_{i+1}-x_i}.$$
\end{thm}
The function $\psi_0$ is a (class-one) $GL(n,\R)$-Whittaker function.
The diffusion process with generator $\mathcal L$ is the analogue of Dyson's Brownian motion in this setting.
The law of the logarithmic partition function $X_1^N(t)=\log Z^N(t)$, 
which can now be seen as the analogue of the
top line (or largest eigenvalue) in the Dyson process, is determined as a corollary.
Similar results have been obtained in~\cite{cosz,osz} for a lattice directed polymer model with 
log-gamma weights which was introduced by Sepp\"al\"ainen~\cite{s}.  In that setting, the eigenfunctions 
of the quantum Toda lattice (also known as Whittaker functions) continue to play a central role, as does 
the geometric lifting of the RSK correspondence which was introduced and studied in the 
papers~\cite{K,NY}.  

Motivated by these developments, in this paper
we introduce continuum versions of the partition functions $Z^N_n(t)$, which we
expect will play an important role in our understanding of the integrable structure which appears to lie behind 
the KPZ and stochastic heat equations.  The continuum partition functions are defined as follows.  
For $n=1,2,\ldots$, $t\ge 0$ and $x,y\in\R$, define
\begin{align}\label{woe}
\Z_n(t,x,y) = p(t,x,y)^n \Big( 1+\sum_{k=1}^\infty \int_{\Delta_k(t)}& \int_{\R^k}
R^{(n)}_k((t_1,x_1),\ldots,(t_k,x_k)) \\
& \times \W(dt_1,dx_1)\cdots \W(dt_k,dx_k) \Big),\nonumber
\end{align}
where $R^{(n)}_k$ is the $k$-point correlation function for a collection of $n$ 
non-intersecting Brownian bridges which all start at $x$ at time $0$ and all end 
at $y$ at time $t$, as defined in section 2 below.
Note that $\Z_1=u$ is the solution of the stochastic heat equation defined by (\ref{chaos}) above.
To explain the above definition we note that, just as in (\ref{gwf}), we can formally write (\ref{woe}) as
\begin{equation}\label{woe1}
\Z_n(t,x,y) = p(t,x,y)^n \E\exp^\diamond \left(\sum_{i=1}^n\int_0^t dW(s,X^i_s) ds\right) ,
\end{equation}
where $(X^1_s,\ldots,X^n_s,\ 0\le s\le t)$ denote the trajectories of $n$ 
non-intersecting Brownian bridges which all start at $x$ at time $0$ and 
all end at $y$ at time $t$.  These should be compared with the partition functions (\ref{sdpfs}).
The first main result of this paper is that the continuum partition functions $Z_n(t,x,y)$ are well-defined.  
\begin{thm}\label{con}
The series (\ref{woe}) is convergent in $L^2(\W)$.
\end{thm}
The proof will given in Section~\ref{proof-con}.
Define $u_1=u$ and, for $n\ge 2$, $u_n=Z_n/Z_{n-1}$.
From Theorem~\ref{qt} we expect that, for each fixed $t>0$, the process
$\mathbf{u}_t(x)=\{ u_n(t,0,x),\ n\in\N \}$, indexed by $x\in\R$, is a diffusion process in $\R^\N$ 
which is a scaling limit at the edge of the diffusion with generator (\ref{pd}), just as the multi-layer Airy process 
introduced in~\cite{ps} is a scaling limit at the edge of Dyson's Brownian motion.  Note that it is not at all clear a priori
that this process should have the Markov property: in fact, this is quite a remarkable property
and can be regarded as the analogue, in this setting, of Pitman's celebrated `$2M-X$' theorem.  
Moreover, for large $t$, the process $\mathbf{u}_t(x),\ x\in\R$
should rescale (after taking logarithms) to the multi-layer Airy process.
At present, we only know this to be the case for the one-dimensional distributions of the first layer:
it has been shown in the papers~\cite{acq,ss4} that the distribution of
$\log[u(t,0,x)/p(t,0,x)]$ (which is independent of $x$) converges in a suitable
scaling limit to the Tracy-Widom distribution.  This result on the first layer
has been tentatively extended to the finite-dimenisonal distributions by
Prolhac and Spohn~\cite{prs1,prs2}.  As such, it is natural to regard the process
$\mathbf{u}_t(x),\ x\in\R $ as the analogue of the multi-layer Airy process in the setting of the
KPZ and stochastic heat equations.  Similarly, for fixed $t>0$, the random sequence 
$\{ u_n(t,0.0),\ n\in\N \}$ can be regarded as the analogue of the Airy point process.  
There are many things to understand in these directions, especially the law of the process 
$(\mathbf{u}_t(x),\ x\in\R)$.  For further recent progress in this direction, see~\cite{bc1,ch}.

It is also natural ask about the evolution of $\mathbf{u}_t$ as $t$ varies: after all, this is an extension
of the process $(u(t,0,\cdot), t>0)$, which evolves according to the stochastic heat equation and (more or less
as an immediate consequence) has the Markov property.  However, from the definition of $\mathbf{u}_t$ 
(which involves non-intersecting Brownian bridges over the time interval $[0,t]$ with the 
{\em same starting and ending points}) there is really no reason to expect the extended process
$(\mathbf{u}_t,\ t> 0)$ to have the Markov property.  Nevertheless, we will present a number of results 
which strongly indicate that it does indeed, somewhat remarkably, have the Markov property.
In fact, we believe that for each $n$, the process 
$$\{ (u_1(t,0,x),\ldots, u_n(t,0,x)),\ t\ge 0\}$$ 
is Markov and give a proof of this claim in the case $n=2$ (see Corollary~\ref{markov2}).

In order to develop a better understanding of the continuum partition functions $Z_n(t,x,y)$,
we consider their analogues when the space-time white-noise potential $\dot{W}(t,x)$ is replaced by 
a smooth time-varying potential $\phi(t,x)$.  For example, we can take $\phi$ to be in the Schwartz 
space $E$ of rapidly decreasing smooth ($C^\infty$) functions on $\R_+\times\R$.
For each $n=1,2,\ldots$, $t>0$ and $x,y\in\R$, define
\begin{equation}\label{fe1}
Z^\phi_n(t,x,y)= p(t,x,y)^n \E\exp\left(\sum_{i=1}^n\int_0^t \phi(s,X^i_s) ds\right) ,
\end{equation}
where once again  $(X^1_s,\ldots,X^n_s,\ 0\le s\le t)$ denote the trajectories of $n$ 
non-intersecting Brownian bridges which all start at $x$ at time $0$ and 
all end at $y$ at time $t$.  On one hand, these are the analogues of the 
partition functions $\Z_n$ introduced above with the white noise $\dot{W}$ 
replaced by the smooth potential $\phi$.  On the other, 
they are directly related to the $\Z_n$ by the formula
\begin{equation}\label{S1}
Z^\phi_n(t,x,y)=\E \left[Z_n(t,x,y) \exp^\w (\W(\phi))\right],
\end{equation}
where 
$$W(\phi)=\int_0^\infty \int_{\R} \phi(s,x)W(ds,dx)$$ and
$\exp^\w (\W(\phi))$ is the Wick exponential of $\W(\phi)$
defined by
\[
\exp^\w({\W} (\phi))= \exp \left( W(\phi) -\frac{1}{2} \int_0^\infty \int_{\R} \phi(s,x)^2 dx ds \right).
\]
In other words, as a function of $\phi$, $Z^\phi_n(t,x,y)$ is the $S$-transform 
of the white noise functional $\Z_n(t,x,y)$ (see, for example,~\cite{houz}). 
In the following, we drop the superscript $\phi$ and write $Z_n=Z_n^\phi$.
Now, as above, we define $u=u_1=Z_1$ and, for $n\ge 2$, $u_n=Z_n/Z_{n-1}$.
Note that, by Feynman-Kac, $u(t,x,y)$ satisfies the heat equation
\begin{equation}\label{he}
\partial_t u=\frac12 \partial_y^2 u+\phi(t,y) u
\end{equation}
with initial condition $u(0,x,y)=\delta(x-y)$.
In this setting we will show, using a generalisation of the Karlin-McGregor formula
(see Propositions~\ref{nib} and \ref{kmg} below), that, for $n\ge 2$,
\begin{equation}\label{kmg-Z}
Z_n(t,x,y)=c_{n,t} \det\left[\partial_x^{i}\partial_y^j u(t,x,y)\right]_{i,j=0}^{n-1} ,
\end{equation}
where $c_{n,t} = t^{n(n-1)/2} (\prod_{j=1}^{n-1} j!)^{-1}$.  But this determinant is the Wronskian
associated with the solutions $u,\partial_x u,\ldots,\partial_x^{n-1}u$ of the heat equation (\ref{he}).
It follows (see for example~\cite{ah}) that the functions $u_n$ are in fact Darboux transformations
and satisfy the coupled system of heat equations
\begin{equation}\label{hs1}
\partial_t u_n=\frac12 \partial_y^2 u_n+[\phi(t,y)+ \partial_y^2 \log \left(Z_{n-1}/p^{n-1}\right)] u_n
\end{equation}
with initial conditions $u_n(0,x,y)=\delta(x-y)$.  These equations are not immediately meaningful if we 
replace the smooth potential $\phi$ by space-time white noise.  However they do suggest that, for each 
$n$, the multi-layer process (in the white noise setting)
$$(\Z_1(t,x,\cdot),\ldots,\Z_n(t,x,\cdot),\ t\ge 0)$$ 
has a Markov evolution.  In order to confirm this, there are two possible directions one could take.
The first, and most obvious one, is to try to make sense of these evolution equations with white-noise
replacing $\phi$.  To this end, it is helpful to consider the following change of variables which, as it happens,
also reveal some deeper structure as we shall see.  Set $\tau_0=1$ and, for $n\ge 1$, define 
$$\tau_n=\det\left[\partial_x^{i}\partial_y^j u(t,x,y)\right]_{i,j=0}^{n-1} ,\qquad a_n=\frac{\tau_{n-1}\tau_{n+1}}{\tau_{n}^2}.$$
We will show (see Propositon~\ref{a-pdes}) that the evolution of the $a_n$ is given by 
\begin{equation}\label{a-pde}
\partial_t a_n = \frac12\partial_y^2 a_n + \partial_y[a_n\partial_y \log u_n].
\end{equation}
It seems to be the case that these evolution equations will make sense as stochastic partial differential equations 
in the white-noise setting~\cite{h}.  

We remark here in passing on an interesting connection to the 2D Toda equations.
We will show (see Lemma~\ref{sylvester}) that $a_n=\partial_{xy}\log \tau_n$, for each $n$.
Thus, if we define,
for $n\ge 1$,
$$q_n=\log (\tau_n/\tau_{n-1}) = \log u_n-\log[t^{n-1}(n-1)!],$$ then $a_n=e^{q_{n+1}-q_n}$ and 
the $q_n$ satisfy the 2D Toda
equations $$\partial_{xy} q_n = e^{q_{n+1}-q_n}-e^{q_n-q_{n-1}},\qquad n\ge 1,$$
with the convention that $q_0=+\infty$.  In this notation, the time-evolution of the $a_n$ is given by
$$\partial_t a_n = \frac12\partial_y^2 a_n + \partial_y[a_n\partial_y q_n].$$

We will also show that there is in fact a second, and quite different, approach one can take in order to 
understand (and prove) the Markov property of the multi-layer process which exploits two quite remarkable
formulas, presented in Theorems~\ref{if} and \ref{kmgw} below, concerning a natural extension of the
partition functions $Z_n$ to collections of non-intersecting Brownian paths which start at distinct points
and end at distinct points.  These are defined as follows.  Set 
\begin{equation}\label{wc1}
\Lambda_n=\{\x=(x_1,\ldots,x_n)\in\R^n:\ x_1\ge\cdots\ge x_n\},
\end{equation} 
and denote the interior of $\Lambda_n$ by $\Lambda^\circ_n$.
For each $t>0$ and for $\x=(x_1,\ldots,x_n)$ and 
$\y=(y_1,\ldots,y_n)$ in $\Lambda^\circ_n$, define
\begin{equation}
K_n(t,\x,\y)= p^*_n(t,\x,\y) \E\exp\left(\sum_{i=1}^n\int_0^t \phi(s,X^i_s) ds\right) ,
\end{equation}
where $(X^1_s,\ldots,X^n_s,\ 0\le s\le t)$ denote the trajectories of a collection of non-intersecting 
Brownian bridges which start at positions $\x=(x_1,\ldots,x_n)$ and end at positions $\y=(y_1,\ldots,y_n)$
at time $t$, and $p^*_n(t,\x,\y)$ is the transition density of a Brownian motion in $\Lambda_n$ killed 
when it first hits the boundary, given by the Karlin-McGregor formula~\cite{km},
$$p^*_n(t,\x,\y) = \sum_{\sigma\in a_n} \sgn(\sigma) \prod_{i=1}^n p(t,x_i,y_{\sigma(i)}).$$
According to  the Feynman-Kac formula, $K_n$ satisfies the equation
\begin{equation}\label{sed1}
\partial_t K_n = \frac12 \Delta_{\y} K_n + \sum_i \phi(t,y_i) K_n
\end{equation}
with Dirichlet boundary conditions on $\partial \Lambda_n$
and initial condition $K_n(0,\x,\y)=\prod_i\delta(x_i-y_i)$.  
In Proposition~\ref{kmg} below we show that the following generalisation of the Karlin-McGregor formula holds:
\begin{equation}\label{kmgw1}
K_n(t,\x,\y) = \det [u(t,x_i,y_j)]_{i,j=1}^{n}.
\end{equation}
Now, define, for $t>0$ and $\x,\y\in\Lambda_n$,
\begin{equation}
M_n(t,\x,\y) = \frac{K_n(t,\x,\y)}{\Delta(\x)\Delta(\y)}.
\end{equation}
This extends continuously to the boundary of $\Lambda_n\times\Lambda_n$;
by Proposition~\ref{kmg}, for $x\in\R$,
\begin{equation}\label{dx1}
M_n(t,x\1 ,\y) = \Delta(\y)^{-1} \det\left[\partial_x^{i-1} u(t,x,y_j)\right]_{i,j=1}^n.
\end{equation}
Rather surprisingly, we will show that the apparently richer object 
$M_n(t,x\1 ,\cdot)$ is, for a fixed $x\in \R$ and $t>0$, 
given as a function of $$(Z_1(t,x,\cdot),\ldots,Z_n(t,x,\cdot)).$$
For $\z\in\Lambda_{n-1}$ and $\y\in\Lambda_n$, write $\z\prec \y$ if 
$y_1\ge z_1> y_2\ge\cdots> y_{n-1}\ge z_{n-1}> y_n$.
For $\y\in\Lambda_n^\circ$, denote by $GT(\y)$ the Gelfand-Tsetlin polytope 
$$\{(y^1,y^2,\ldots,y^{n-1})\in
\Lambda_1\times\Lambda_2\times\cdots\times\Lambda_{n-1}:\ \
y^1\prec y^2\prec\cdots\prec y^{n-1}\prec y\}.$$
Then (see Theorem~\ref{if}) for $t>0$, $x\in\R$ and $\y\in\Lambda_n^\circ$,
\begin{equation}\label{if1}
M_n(t,x\1 ,\y)= \Delta(\y)^{-1} \prod_{i=1}^n u(t,x,y_i)
\int_{GT(\y)} \prod_{k=1}^{n-1} \prod_{i=1}^{n-k} a_k(t,x,y^{n-k}_i) dy^{n-k}_i .
\end{equation}
In the case $\phi=0$, this reduces to the fact that the volume of $GT(\y)$ 
is proportional to $\Delta(\y)$.  Now, if the analogous formula holds in the white-noise 
setting, with $K_n(t,\x,\y)$ defined by the chaos series (\ref{WOE}) below, then
this imply the Markov property of the multi-layer process
$$(\Z_1(t,x,\cdot),\ldots,\Z_n(t,x,\cdot),\ t\ge 0).$$ 
This is because the process for each $x\in\R$ and for each $n$, the process 
$(M_n(t,x\1 ,\cdot), t\ge 0)$ clearly has the Markov property.  However, an
important ingredient in this argument is the Karlin-McGregor formula (\ref{kmgw1}).
A priori, there is no reason to expect this to hold in the white-noise setting.
In fact, the natural analogue of this formula in the white-noise setting is the formula
$$K_n(t,\x,\y) = \det {}^\diamond [u(t,x_i,y_j)]_{i,j=1}^{n},$$
where $\det^\diamond$ indicates that the products in the expansion of the determinant
are {\em Wick products}.  Nevertheless, quite remarkably, we will show in Section~\ref{KMG}
that the formula (\ref{kmgw1}) does in fact hold in the white-noise setting.  In Section~\ref{evol}
we argue that, modulo technical considerations, this indicates that the analogue of the integral 
formula (\ref{if1}) should hold in the white-noise setting.  We give a proof in the case $n=2$,
just to demonstrate that it can be done.  The main technical issue concerns the continuity of
$M_n(t,\x,\y)$ at the boundary of $\Lambda_n\times\Lambda_n$.  A proof of the existence of 
an almost surely continuous extension to the boundary based on Kolmogorov's criterion would 
be long and technical and beyond the scope of the present work.  Here we satisfy ourselves with 
a continuous extension in  $L^2$, which allows us to establish the analogue of the integral formula 
(\ref{if1}), and hence the Markov property of the multi-layer process, in the special case $n=2$.
We remark that an interesting consequence of our proof of the $L^2$-continuity property is
that the ratio of two solutions to the stochastic heat equation is in $H^1$.  In fact, such ratios have 
recently been shown (in a slightly different setting, for smooth initial data and periodic boundary
conditions) by Hairer~\cite{h} to be in $C^{3/2-\epsilon}$.  The index $3/2$ is consistent with our 
expectation that the continuum partition functions $\Z_n(t,x,y)$ are locally Brownian in the space variable $y$.

We conclude this section with some remarks on the RSK interpretation.
As remarked above, the multi-layer construction presented in this
paper is based on a geometric lifting of the RSK correspondence, so it is 
natural to consider such an interpretation in the continuum setting.  
The analogue of RSK in the context of smooth potentials is the mapping
$$\phi\big|_{[0,t]\times\R} \mapsto \{u_n(t,0,\cdot),\ n\ge 1\}.$$
In the language of RSK, $\{u_n(t,0,x),\ n\ge 1;\ x\ge 0\}$ is the $P$-tableau,
$\{u_n(t,0,-x),\ n\ge 1;\ x\ge 0\}$ is the $Q$-tableau, and their common `shape'
is the sequence $\{u_n(t,0,0),\ n\ge 1\}$.  
We note the following symmetry, which corresponds to a 
well known symmetry property of the RSK correspondence.
Writing $f=\phi\big|_{[0,t]\times\R}$, 
$P(f)=\{u_n(t,0,x),\ n\ge 1;\ x\ge 0\}$,
$Q(f)=\{u_n(t,0,-x),\ n\ge 1;\ x\ge 0\}$ and $f^\dagger(s,x)=f(s,-x)$, we have:
$P(f^\dagger)=Q(f)$ and $Q(f^\dagger)=P(f)$.
Similarly, in the white noise setting, we define 
$$P(\W_{[0,t]})=\{\U_n(t,0,x),\ n\ge 1;\ x\ge 0\}$$
and
$$Q(\W_{[0,t]})=\{\U_n(t,0,-x),\ n\ge 1;\ x\ge 0\}$$ where
$\W_{[0,t]}$ denotes the restriction of $\W$ to $[0,t]\times\R$.
As explained above, we expect that, for each $t>0$,
$P(\W_{[0,t]})$ and $Q(\W_{[0,t]})$ are diffusion processes in $\R^\N$ (indexed
by $x\ge 0$) which are conditionally independent given their starting position 
$\{\U_n(t,0,0),\ n\ge 1\}$.  This would be the analogue, in this setting,
of Pitman's `$2M-X$' theorem.  Since the first draft of the present paper appeared,
substantial progress has been made in this direction by Corwin and Hammond~\cite{ch},
where a natural candidate for this infinite-dimensional diffusion process has been
constructed.

The outline of the paper is as follows.  In the next section we provide some background 
on non-intersecting Brownian motions and their bridges. Following this we study the analogue
 of the partition functions when the space-time white noise is replaced  by a smooth time-varying potential. 
 In this setting we establish a connection with Darboux 
transformations of solutions to the heat equation, which give rise to evolution equations for the multi-layer 
process of partition functions.
 These equations are not directly meaningful in the white noise setting, but suggest that the multi-layer process 
 has a Markovian evolution.  We also give proofs of the integral formula~(\ref{if1}) above, the evolution
 equations (\ref{a-pde}) and remark on the connection with the 2D Toda equations.  In Section~\ref{proof-con}
 we present the proof of Theorem~\ref{con} on the existence of the continuum partition functions in the white-noise setting.
In Section~\ref{KMG}, we show that the Karlin-McGregor formula (\ref{kmgw}) holds in the white-noise setting.
In Section~\ref{evol} we consider the evolution of the multi-layer process in the white-noise setting and 
give a proof of the Markov property for $n=2$.

\bigskip

 {\em Acknowledgements.}  Thanks to Ivan Corwin, Martin Hairer, Jeremy Quastel,
Gregorio Moreno-Flores and Roger Tribe for helpful discussions.  This research was 
supported in part by EPSRC grant EP/I014829/1.

\section{Non-intersecting Brownian motions}

Non-intersecting Brownian motions play a large role in this paper, and we record here definitions and facts concerning  them that will be useful to us.

Recall that
\begin{equation}\label{wc}
\Lambda_n=\{\x=(x_1,\ldots,x_n)\in\R^n:\ x_1\ge\cdots\ge x_n\},
\end{equation} and denote the interior of $\Lambda_n$ by $\Lambda^\circ_n$. Standard $n$-dimensional Brownian motion killed on exiting $\Lambda^\circ_n$ has transition densities given by the Karlin-McGregor forumla \cite{km},
 $$p^*_n(t,\x,\y) = \sum_{\sigma\in a_n} \sgn(\sigma) \prod_{i=1}^n p(t,x_i,y_{\sigma(i)}).$$
 Dyson Brownian motion is obtained as a Doob-$h$ transform of this killed process by the harmonic function
$\Delta(\x)=\prod_{i<j} (x_i-x_j)$.
It is self-dual relative the measure $ \Delta(\y)^2d\y$ 
and has transition densities (with respect to this measure) given by
\begin{equation}
\label{bridgedensity}
q_n(t,\x,\y)= \frac {p^*_n(t,\x,\y)}{\Delta(\x) \Delta(\y)}.
\end{equation}
\begin{lem}\label{qn}
For each $t>0$, the transition density $q_n(t,\x,\y)$ extends continuously to a 
uniformly bounded strictly positive function on 
$\Lambda_n \times \Lambda_n$.
\end{lem}
\begin{proof}
By the Harish-Chandra/Itzykson-Zuber formula, we can write
\begin{equation}\label{hciz}
q_n(t,x,y)=(2\pi)^{-n/2} t^{-n^2/2} c_n \int_{U(n)} e^{-\mbox{tr}(X-UYU^*)^2/2t} dU,
\end{equation}
where $1/c_n=\prod_{j=1}^{n-1}j!$, $X$ and $Y$ are diagonal matrices with entries given 
by the vectors $\x$ and $y$, and the integral is with respect to normalised Haar measure
on the group of $n\times n$ unitary matrices.
By bounded convergence, the RHS defines a continuous function on $\Lambda_n \times \Lambda_n$,
and is bounded by $(2\pi)^{-n/2} t^{-n^2/2} c_n$.  The strict positivity follows from the strict positivity
of the integrand.
\end{proof}
The semigroup property
\[
q_n(s+t,\x, \z)=\int_{\Lambda_n} q_n(s, \x,\y)q_n(t,\y,\z) \Delta(\y)^2d\y,
\]
thus also extends to the boundary, by continuity and dominated convergence.
Moreover, it follows from the representation~\eqref{hciz} that
$q_n(t,\x,\y)\Delta(\y)^2 d\y$ defines a probability measure on $\Lambda_n$
for every $t>0$ and $x\in\Lambda_n$.
Consequently, Dyson Brownian motion can be started from any point on the boundary of $\Lambda_n$.
In fact, as was shown by C\'epa and L\'epingle~\cite{cl}, it almost surely never subsequently 
returns to the boundary.

 Let $(H_t, t \geq 0)$ be a standard Brownian motion in the space of $n\times n$  Hermitian matrices. Then the vector of ordered  real-valued eigenvalues  of $H_t$ evolves as Dyson Brownian motion. This can be verified by deriving the transition density for the eiqenvalues using the Harish-Chandra formula. Alternatively, following Dyson's original approach, applying It\^{o}'s formula shows that, denoting  the vector of eigenvalues  by $(X^1_t, X^2_t,\ldots, X^n_t)$, the following stochastic differential equations are satisfied.
 \begin{equation}
\label{sdes}
X^i_t= X^i_0+\beta^i_t+\sum_{j\ne i} \int_0^t \frac{ds}{X^i_s-X^j_s} , 
\end{equation}
where $\beta^i$, $i=1,2,\ldots,n$ are a collection
of independent standard one-dimensional Brownian motions. Note that these equations hold even if  the intitial value $H_0$ of the Hermitian Brownian motion has repeated  eigenvalues, in which case the Dyson Brownian motion is starting from the boundary of $\Lambda_n$. This can been seen by the following argument. Applying It\^{o}'s formula from  some strictly positive time $\epsilon$ onwards ( recalling the process of eigenvalues does not visit the boundary) we obtain for $t\geq \epsilon$,
\begin{equation}
\label{sde1}
X^i_t= X^i_\epsilon+\beta^{i,\epsilon}_t+\sum_{j\ne i} \int_\epsilon^t \frac{ds}{X^i_s-X^j_s} , 
\end{equation}
where $\beta^{i,\epsilon}$ are Brownian motions. Now for $\epsilon<\epsilon^\prime$ the increments of $\beta^{i, \epsilon}_{\epsilon^\prime-\epsilon+\cdot}$ and $\beta^{i, \epsilon^\prime}$ agree, and by virtue of this consistency there exist  Brownian motions  $\beta^i$  starting from $0$ such that $\beta^{i, \epsilon}_t= \beta^i_{\epsilon+t}-\beta^i_\epsilon$ for all $\epsilon>0$. Now returning to \eqref{sde1}, writing it using $\beta^i$, and letting $\epsilon$ tend down to $0$ gives \eqref{sdes} as desired even in the case when  the Dyson Brownian motion starts from the boundary.

We can construct bridges for Dyson Brownian motion using the standard Markovian framework, see for example Proposition 1 of \cite{fitz}.  Specifically given points $\x$ and $\y$ belonging to $\Lambda_n$, we define the bridge from $\x$  at time $0$ ending at $\y$ at time $t$, to be a process $({\bf X}_s, 0\leq s\leq t)$ whose law over $[0,s]$, for any $s<t$, is absolutely continuous with respect to that of Dyson Brownian motion starting from $\x$, with a density
\begin{equation}\label{bridge}
\frac{q_n( t-s, {\bf X}_s,\y)}{q_n(t,\x,\y)}.
\end{equation}
Note that this is well-defined as the denominator is strictly positive, by Lemma~\ref{qn} above.
We will also refer to this bridge, somewhat informally,  as a collection of non-intersecting Brownian bridges.  
In the special case $\x=x\1$, $\y=y\1$, where $x,y\in\R$ and we denote by $\1\in\R^n$ the vector with all coordinates 
equal to 1, the process is also often referred to as a watermelon.
The correlation function $R^{(n)}_k((t_1,x_1), \ldots, (t_k,x_k))$ appearing in the definition \eqref{woe} is
defined to be the sum over $i_1, i_2, \ldots,i_k$ of the (continuous) probability densities of $(X^{i_1}_{t_1}, \ldots, X^{i_k}_{t_k})$ 
with respect to Lebesgue measure evaluated at $(x_1, x_2, \ldots,x_k)$.  
Correlation functions for non-intersecting Brownian bridges with arbitrary starting and ending positions,
which appear in Section 5 below, are defined analogously.

\section{Darboux transformations }

In this section we replace the white noise potential $\dot{W}$ by a smooth potential $\phi$,
which we assume for convenience to be in the Schwartz space $E$ of rapidly 
decreasing smooth ($C^\infty$) functions on $\R_+\times\R$.

For each $n=1,2,\ldots$, $t>0$ and $x,y\in\R$, define
\begin{equation}\label{fe}
Z^\phi_n(t,x,y)= p(t,x,y)^n \E\exp\left(\sum_{i=1}^n\int_0^t \phi(s,X^i_s) ds\right) ,
\end{equation}
where $(X^1_s,\ldots,X^n_s),\ 0\le s\le t,$ denote the trajectories of $n$ 
non-intersecting Brownian bridges which all start at $x$ at time $0$ and 
all end at $y$ at time $t$.  On one hand, these are the analogues of the 
partition functions $\Z_n$ introduced in the previous section with the
white noise $\dot{W}$ replaced by a smooth potential $\phi$.  On the other, 
they are directly related to the $\Z_n$ by the formula
\begin{equation}\label{S}
Z^\phi_n(t,x,y)=\E \left[Z_n(t,x,y) \exp^\w (\W(\phi))\right],
\end{equation}
where 
$$W(\phi)=\int_0^\infty \int_{\R} \phi(s,x)W(ds,dx)$$ and
$\exp^\w (\W(\phi))$ is the Wick exponential of $\W(\phi)$
defined by
\[
\exp^\w({\W} (\phi))= \exp \left( W(\phi) -\frac{1}{2} \int_0^\infty \int_{\R} \phi(s,x)^2 dx ds \right).
\]
In other words, as a function of $\phi$, $Z^\phi_n(t,x,y)$ is the $S$-transform 
of the white noise functional $\Z_n(t,x,y)$ (see, for example,~\cite{houz}). 
To see that \eqref{S} holds, on the RHS replace $\Z_n(t,x,y)$ by the 
series \eqref{woe} and $\exp^\w (W(\phi))$ by its Wiener chaos expansion;
computing the expectation of the product of these two series we obtain
\begin{align*}
p(t,x,y)^n \sum_{k=0}^\infty \int_{\Delta_k(t)} \int_{\R^n} \phi(t_1,x_1) & \ldots \phi(t_k,x_k) \\
 \times R^{(n)}_k((t_1,x_1), \ldots,& (t_k,x_n))  dx_1 \ldots dx_k dt_1\ldots dt_k\\
& =p(t,x,y)^n  \E\exp\left(\sum_{i=1}^n\int_0^t \phi(s,X^i_s) ds\right) .
 \end{align*}

For the remainder of this section we will only consider the case of smooth 
potential $\phi$.  For notational convenience we will drop the superscript
and simply write $Z_n(t,x,y)=Z^\phi_n(t,x,y)$.
By the Feynman-Kac formula, $u:=Z_1$ satisfies the heat equation
\begin{equation}
\partial_t u=\frac12 \partial_y^2 u+\phi(t,y) u
\end{equation}
with initial condition $u(0,x,y)=\delta(x-y)$.
\begin{prop}\label{nib} For $n\ge 2$,
\begin{equation}
Z_n(t,x,y)=c_{n,t} \det\left[\partial_x^{i}\partial_y^j u(t,x,y)\right]_{i,j=0}^{n-1} ,
\end{equation}
where $c_{n,t} = t^{n(n-1)/2}c_n$ and $1/c_n= \prod_{j=1}^{n-1} j!$.
\end{prop}
We will prove this via a generalisation of the Karlin-McGregor formula.

For each $t>0$ and for $\x=(x_1,\ldots,x_n)$ and 
$\y=(y_1,\ldots,y_n)$ in $\Lambda^\circ_n$, define
\begin{equation}\label{tildeZphi}
K_n(t,\x,\y)= p^*_n(t,\x,\y) \E\exp\left(\sum_{i=1}^n\int_0^t \phi(s,X^i_s) ds\right) ,
\end{equation}
where $(X^1_s,\ldots,X^n_s,\ 0\le s\le t),$ denote the trajectories of a collection of non-intersecting 
Brownian bridges which start at positions $x_1,\ldots,x_n$ and end at positions $y_1,\ldots,y_n$
at time $t$, and $p^*_n(t,\x,\y)$ is the transition density of a Brownian motion in $\Lambda_n$ killed 
when it first hits the boundary, given by the Karlin-McGregor formula.

\begin{prop}\label{kmg}
\begin{equation}
K_n(t,\x,\y) = \det [u(t,x_i,y_j)]_{i,j=1}^{n}.
\end{equation}
\end{prop}
\begin{proof}
According to  the Feynman-Kac formula, $K_n$ satisfies the equation
\begin{equation}\label{sed}
\partial_t K_n = \frac12 \Delta_{\y} K_n + \sum_i \phi(t,y_i) K_n
\end{equation}
with Dirichlet boundary conditions on $\partial \Lambda_n$
and initial condition $K_n(0,\x,\y)=\prod_i\delta(x_i-y_i)$.  
Moreover it is the unique solution to this initial-boundary value 
problem which vanishes as $|\y|\to\infty$ uniformly for $t$ in compact intervals. 
This follows from a variant of the maximum principle 
(see, for example, \cite[Chapter 2, Theorem 2]{f}), 
which applies in this setting since $\phi$ is bounded and continuous.

On the other hand, $\det [Z(t,x_i,y_j)]_{i,j=1}^{n}$
satisfies the same initial-boundary value problem
and vanishes as $|\y|\to\infty$ uniformly for $t$ in compact intervals. 
So the identity follows by uniqueness. 
\end{proof}

We remark that the same argument can be applied to more general expressions than \eqref{tildeZphi} in which the potential $ (s,\x) \mapsto \sum \phi(s, x_i)$ is replaced by a bounded continuous potential $\psi(t,\x)$ which is a symmetric function of the coordinates of $\x$; we will make use of this fact in the proof of Theorem \ref{kmgw} below.

\begin{proof}[Proof of Proposition \ref{nib}.]
It is immediate from the definitions that
$$\frac{Z_n(t,a,b)}{p(t,a,b)^n}=\lim_{\x\to a\1 ,\y\to b\1 } \frac{K_n(t,\x,\y)}{p^*_n(t,\x,\y)}.$$
Now 
$$\lim_{\x\to a\1 ,\y\to b\1 } \frac{p^*_n(t,\x,\y)}{\Delta(\x)\Delta(\y)}=  p(t,a,b)^n t^{-n(n-1)/2} c_n,$$
where $\Delta(\x)=\prod_{i<j} (x_i-x_j)$. This can be inferred, for example, from \cite[Lemma 5.11]{bbo2}).
On the other hand, by Proposition \ref{kmg},
$$\lim_{\x\to a\1 ,\y\to b\1 } \frac{K_n(t,\x,\y)}{\Delta(\x)\Delta(\y)}=c_n^2
\det\left[\partial_a^{i}\partial_b^j u(t,a,b)\right]_{i,j=0}^{n-1}.$$
Given that we have already established that this limit exists, this follows, for example, from~\cite[Theorem 15]{w}), 
where the above formula is given in the case $\x=a\1+\epsilon\delta$ and $\y=b\1+\epsilon\delta$, where 
$\delta=(n-1,\ldots,1,0)$ and $\epsilon\to 0$.

\end{proof}

Define $u_n(t,x,y)$ recursively by $Z_n=u_1 u_2 \cdots u_n$.
\begin{prop}\label{mlshe}
The functions $u_n$ satisfy the coupled system of heat equations
\begin{equation}\label{hs}
\partial_t u_n=\frac12 \partial_y^2 u_n+[\phi(t,y)+ \partial_y^2 \log \left(Z_{n-1}/p^{n-1}\right)] u_n
\end{equation}
with initial conditions $u_n(0,x,y)=\delta(x-y)$ and the convention $Z_0=1$.
\end{prop}
The equations \eqref{hs} follow from Proposition \ref{nib} together with known properties of Darboux
transformations of solutions to one-dimensional heat equations with time-varying potentials, 
see for example~\cite{ah,ah1}.  For completeness, we will include a direct proof of Proposition~\ref{mlshe}
just after the statement of Proposition~\ref{a-pdes} below.

The coupled heat equations of Proposition \ref{mlshe} are not immediately meaningful if we replace 
the smooth potential $\phi$ by space-time white noise.
However they do suggest that the multi-layer process (in the white noise setting)
$$(\Z_1(t,x,\cdot),\ldots,\Z_n(t,x,\cdot),\ t\ge 0)$$ 
has the Markov property. 
In the following, we introduce a natural extension of the $Z_n$ which will play an important 
role in our understanding of the Markov property when we return to the white noise setting.

Define, for $t>0$ and $\x,\y\in\Lambda_n$,
\begin{equation}
M_n(t,\x,\y) = \frac{K_n(t,\x,\y)}{\Delta(\x)\Delta(\y)}.
\end{equation}
This extends continuously to the boundary of $\Lambda_n\times\Lambda_n$;
by Proposition~\ref{kmg}, for $x\in\R$,
\begin{equation}\label{dx}
M_n(t,x\1 ,\y) = \Delta(\y)^{-1} \det\left[\partial_x^{i-1} u(t,x,y_j)\right]_{i,j=1}^n.
\end{equation}
Rather surprisingly, we will now show that the apparently richer object 
$M_n(t,x\1 ,\cdot)$ is, for a fixed $x\in \R$ and $t>0$, 
given as a function of $$(Z_1(t,x,\cdot),\ldots,Z_n(t,x,\cdot)).$$

Recall from Proposition \ref{nib} that, for $n\ge 2$,
$$Z_n(t,x,y)=c_{n,t} \det\left[\partial_x^{i}\partial_y^j u(t,x,y)\right]_{i,j=0}^{n-1} ,$$
where $c_{n,t} = t^{n(n-1)/2} c_n$.
Let us write 
\begin{equation}\label{tau}
\tau_n(t,x,y)=\det\left[\partial_x^{i}\partial_y^j u(t,x,y)\right]_{i,j=0}^{n-1} .
\end{equation}
For notational convenience, set $\tau_0=Z_0=1$ and $\Lambda_1=\R$.
For $n\ge 1$, define
\begin{equation}\label{a}
a_n=\frac{\tau_{n-1}\tau_{n+1}}{\tau_{n}^2}=\frac{n}{t}\frac{Z_{n-1}Z_{n+1}}{Z_{n}^2}.
\end{equation}
Here we are using the fact that $c_{n-1,t}c_{n+1,t}/c_{n,t}^2=t/n$.

For $\z\in\Lambda_{n-1}$ and $\y\in\Lambda_n$, write $\z\prec \y$ if 
$y_1\ge z_1> y_2\ge\cdots> y_{n-1}\ge z_{n-1}> y_n$.
For $\y\in\Lambda_n^\circ$, denote by $GT(\y)$ the Gelfand-Tsetlin polytope 
$$\{(y^1,y^2,\ldots,y^{n-1})\in
\Lambda_1\times\Lambda_2\times\cdots\times\Lambda_{n-1}:\ \
y^1\prec y^2\prec\cdots\prec y^{n-1}\prec y\}.$$
\begin{thm} \label{if}
For $t>0$, $x\in\R$ and $\y\in\Lambda_n^\circ$,
$$M_n(t,x\1 ,\y)= \Delta(\y)^{-1} \prod_{i=1}^n u(t,x,y_i)
\int_{GT(\y)} \prod_{k=1}^{n-1} \prod_{i=1}^{n-k} a_k(t,x,y^{n-k}_i) dy^{n-k}_i .$$
\end{thm}
In the case $\phi=0$, this reduces to the fact that the volume of $GT(\y)$ 
is proportional to $\Delta(\y)$.  By (\ref{dx}) and Proposition \ref{nib},
this theorem can be seen as a consequence of the next two lemmas.
\begin{lem}\label{cb} 
If $f_1,f_2,\ldots$ is a sequence of continuously
differentiable functions on $\R$ with $f_1\equiv 1$ then
$$\det[f_i(y_j)]_{i,j=1}^n = \int_{z\prec y} \det[f_{i+1}'(z_j)]_{i,j=1}^{n-1} dz_1\ldots dz_{n-1}.$$
\end{lem}
\begin{proof} Using the formula 
$$1_{\z\prec \y} = \det\left[ 1_{y_{j+1}<z_i\le y_j}\right]_{i,j=1}^{n-1}$$
we have, by a generalisation of the Cauchy-Binet formula (see, for example,\cite{j}),
\begin{align*}
 \int_{\z\prec \y} &\det[f_{i+1}'(z_j)]_{i,j=1}^{n-1} dz_1\ldots dz_{n-1}\\
 &=\int_{\Lambda_{n-1}} \det[f_{i+1}'(z_j)]_{i,j=1}^{n-1}  \det\left[ 1_{y_{j+1}<z_i\le y_j}\right]_{i,j=1}^{n-1} dz_1\ldots dz_{n-1}\\
 &=\det\left[ \int_{y_{j+1}}^{y_j} f_{i+1}'(z) dz\right]_{i,j=1}^{n-1} 
 =\det\left[ f_{i+1}(y_j)-f_{i+1}(y_{j+1}) \right]_{i,j=1}^{n-1} \\
 &=\det[f_i(y_j)]_{i,j=1}^n,
 \end{align*}
 as required.
\end{proof}
\begin{lem}\label{sylvester}
Let $g(x,y)$ be a smooth function and define $W_0=1$, $W_1=g$
and, for $n\ge 2$, $W_n=\det\left[\partial_x^{i}\partial_y^j g(x,y)\right]_{i,j=0}^{n-1}$.
Suppose that $W_n$ is strictly positive for all $n\ge 1$ and define
$T_n=W_{n-1}W_{n+1}/W_{n}^2$.
Then the following identities hold:
$$T_n=\partial_{xy}\log W_n=\partial_y (\partial_y (\cdots \partial_y (\partial_x^n g/g)/T_1)/T_2)\cdots)/T_{n-1}),$$
$$\det[\partial_x^{i-1} g(x,y_j)]_{i,j=1}^n = \prod_{i=1}^n g(x,y_i)
\int_{GT(y)} \prod_{k=1}^{n-1} \prod_{i=1}^{n-k} T_k(x,y^{n-k}_i) dy^{n-k}_i .$$
\end{lem}
\begin{proof}  
From well-known properties of Wronskian determinants, $\partial_xW_n$,
$\partial_yW_n$ and $\partial_{xy}W_n$ can be expressed as determinants, namely
$$\partial_xW_n=\det\left[\partial_x^{i}\partial_y^j g(x,y)\right]_{i=0,1,\ldots,n-2,n;\ j=0,\ldots,n-1},$$
$$\partial_yW_n=\det\left[\partial_x^{i}\partial_y^j g(x,y)\right]_{i=0,\ldots,n-1;\ j=0,1,\ldots,n-2,n},$$
$$\partial_{xy}W_n=\det\left[\partial_x^{i}\partial_y^j g(x,y)\right]_{i=0,1,\ldots,n-2,n;\ j=0,1,\ldots,n-2,n}.$$
It follows from Sylvester's determinant identity~\cite[p22]{hj} that
$$W_n\partial_{xy}W_n-(\partial_xW_n)(\partial_y W_n)=W_{n-1}W_{n+1},$$
proving the first identity.  Essentially the same argument shows that, for any $k\ge 1$,
$$W_n\partial_x^k\partial_yW_n-(\partial_x^kW_n)(\partial_yW_n)=W_{n-1}\partial_x^{k-1}W_{n+1}.$$
This implies that
$$(\partial_y(\partial_x^kW_n/W_n))/T_n=(\partial_x^{k-1}W_{n+1})/W_{n+1}.$$
In particular,
$$(\partial_y(\partial_x^ng/g))/T_1=(\partial_x^{n-1}W_2)/W_2,$$
$$(\partial_y(\partial_x^{n-1}W_2/W_2))/T_2=(\partial_x^{n-2}W_3)/W_3,$$
and so on, yielding the second identity.

Now, by Lemma~\ref{cb},
\begin{align*}
\det&\left[\frac{\partial_x^{i-1} g(x,y_j)}{g(x,y_j)}\right]_{i,j=1}^{n}
= \int_{y^{n-1}\prec y} 
\det\left[\partial_{y^{n-1}_j} \frac{\partial_x^i g(x,y^{n-1}_j)}{g(x,y^{n-1}_j)}\right]_{i,j=1}^{n-1}
 \prod_{i=1}^{n-1} dy^{n-1}_i\\
 &= \int_{y^{n-1}\prec y} 
\det\left[ \frac{\partial_{y^{n-1}_j} \frac{\partial_x^i g(x,y^{n-1}_j)}{g(x,y^{n-1}_j)} }{ T_1(x,y_j^{n-1})}\right]_{i,j=1}^{n-1}
 \prod_{i=1}^{n-1} T_1(x,y^{n-1}_i)  dy^{n-1}_i.
 \end{align*}
Applying Lemma~\ref{cb} again, using $T_1=\partial_{y} (\partial_x g/g)$, we obtain
\begin{align*}
\det & \left[ \frac{\partial_{y^{n-1}_j} \frac{\partial_x^i g(x,y^{n-1}_j)}{g(x,y^{n-1}_j)} }{ T_1(x,y_j^{n-1})}\right]_{i,j=1}^{n-1}\\
&= \int_{y^{n-2}\prec y^{n-1}}
\det\left[\partial_{y_j^{n-2}} \frac{\partial_{y^{n-2}_j} \frac{\partial_x^{i+1} g(x,y^{n-2}_j)}{g(x,y^{n-2}_j)} }{ T_1(x,y_j^{n-2})}
\right]_{i,j=1}^{n-2}  \prod_{i=1}^{n-2} dy^{n-2}_i \\
&=\int_{y^{n-2}\prec y^{n-1}}
\det\left[\frac{ \partial_{y_j^{n-2}} \frac{\partial_{y^{n-2}_j} \frac{\partial_x^{i+1} g(x,y^{n-2}_j)}{g(x,y^{n-2}_j)} }{ T_1(x,y_j^{n-2})}}
{T_2(x,y^{n-2}_j)}
\right]_{i,j=1}^{n-2}  \prod_{i=1}^{n-2} T_2(x,y^{n-2}_i) dy^{n-2}_i.
\end{align*}
Now apply Lemma~\ref{cb} again, using $T_2=\partial_{y} (\partial_{y} (\partial_x^2 g/g)/ T_1)$,
and so on, to obtain the third identity.
\end{proof}

Note that, by Lemma~\ref{sylvester}, 
\begin{equation}\label{a-toda} a_n=\partial_{xy}\log \tau_n.\end{equation} 
Thus, if we define, for $n\ge 1$,
$$q_n=\log (\tau_n/\tau_{n-1}) = \log u_n-\log[t^{n-1}(n-1)!],$$ then $a_n=e^{q_{n+1}-q_n}$ 
and, by \eqref{a-toda}, the functions $q_n$ satisfy the 2D Toda equations 
$$\partial_{xy} q_n = e^{q_{n+1}-q_n}-e^{q_n-q_{n-1}},\qquad n\ge 1,$$
with the convention that $q_0=+\infty$. 

The time-evolution of the $a_n$ is given by the following proposition.
\begin{prop}\label{a-pdes}  For $n\ge 1$,
\begin{equation}\label{an-pde}
\partial_t a_n = \frac12\partial_y^2 a_n + \partial_y[a_n\partial_y \log u_n].
\end{equation}
\end{prop}
\begin{proof}[Proof of Propositions \ref{mlshe} and \ref{a-pdes}.]
First we note that the initial condition $u_n(0,x,y)=\delta(x-y)$ follows 
immediately from the definition \eqref{fe} of $Z_n$.
We will verify the equations \eqref{hs} and \eqref{an-pde} simultaneously 
by induction over $n$.
Write $h_n=\log u_n$ and note that \eqref{hs} is equivalent to
$$\partial_t h_n = \frac12 \partial_y^2 h_n + \frac12 (\partial_y h_n)^2
+\phi(t,y)+ \partial_y^2 \log\left(Z_{n-1}/p^{n-1}\right).$$
For $n=1$, the equation \eqref{hs} holds by hypothesis.  Thus
$$\partial_t h_1 = \frac12 \partial_y^2 h_n + \frac12 (\partial_y h_n)^2+\phi(t,y).$$
It follows that $a_1=\partial_{xy}h_1$ satisfies
$$\partial_t a_1 = \frac12\partial_y^2 a_1 + \partial_y[a_1\partial_y h_1],$$
as required.  

Now assume the induction hypothesis (with two parts):
$$\partial_t u_n=\frac12 \partial_y^2 u_n+[\phi(t,y)+ \partial_y^2 \log \left(Z_{n-1}/p^{n-1}\right)] u_n$$
and
$$\partial_t a_n = \frac12\partial_y^2 a_n + \partial_y[a_n\partial_y h_n].$$
Then, using $u_{n+1}= nt a_n u_n$ and $\partial_y^2\log (1/p^n)=n/t$,
\begin{eqnarray*}
\partial_t u_{n+1} &=& n a_n u_n + nt [a_n \partial_t u_n+u_n\partial_t a_n]\\
&=& \frac1t u_{n+1} + nt a_n [\frac12 \partial_y^2 u_n [\phi(t,y)+ \partial_y^2 \log \left(Z_{n-1}/p^{n-1}\right)] u_n]\\
&& + nt u_n [ \frac12\partial_y^2 a_n + \partial_y[a_n\partial_y h_n] ]\\
&=& \frac12\partial_y^2 u_{n+1} + [\phi(t,y)+ \partial_y^2 \log \left(Z_{n-1}/p^{n-1}\right) + \partial_y^2 h_n +\frac1t] u_{n+1}\\
&=& \frac12\partial_y^2 u_{n+1} + [\phi(t,y)+ \partial_y^2 \log \left(Z_{n}/p^{n}\right)] u_{n+1},\\
\end{eqnarray*}
as required.  Note that this implies
$$\partial_t h_{n+1} = \frac12 \partial_y^2 h_{n+1} + \frac12 (\partial_y h_{n+1})^2
+\phi(t,y)+ \partial_y^2 \log\left(Z_n/p^n\right).$$
By \eqref{a-toda} we can write $a_{n+1}=a_n+\partial_{xy} h_{n+1}$.
Thus, using the second part of the induction hypothesis again,
$$\partial_t a_{n+1} = \frac12\partial_y^2 a_{n+1} + \partial_y[a_n\partial_y h_n]
+\partial_y[ (a_{n+1}-a_n) \partial_y h_{n+1} ] +\partial_y^2 a_n.$$
But
$$\partial_y^2 a_n = \partial_y[a_n\partial_y\log a_n] = \partial_y[a_n(\partial_y h_{n+1}-\partial_y h_n)],$$
so we have
$$\partial_t a_{n+1} = \frac12\partial_y^2 a_{n+1} + \partial_y[a_{n+1}\partial_y h_{n+1}],$$
as required.
\end{proof}

\section{Proof of Theorem~\ref{con}}\label{proof-con}

We return now to the white noise setting, denoting by $u(t,x,y)$ the solution to the 
stochastic heat equation (\ref{she1}) with initial condition $u(0,x,y)=\delta(x-y)$.
In this section we will show that for each $n\ge 2$, $Z_n(t,x,y)$
defined by (\ref{woe}) is convergent in $L^2(W)$ or, equivalently,
\begin{equation} \label{squares}
\sum_{k=0}^\infty \int_{\Delta_k(t)} \int_{\R^k}
R^{(n)}_k((t_1,x_1),\ldots,(t_k,x_k))^2 dx_1\cdots dx_k dt_1\cdots dt_k <\infty .
\end{equation}
The first step is to show that this is equivalent to  $\E e^{L} <\infty$, where $L$
is the total intersection local time between two independent copies 
of the system of $n$ non-intersecting Brownian bridges.
Let $(X^i_s,\ 0\le s\le t,\ i=1,\ldots,n)$ 
 be a collection of non-intersecting bridges which all start at $x$ at 
time $0$ and all end at $y$ at time $t$, and let and $(Y^i_s,\ 0\le s\le t,\ i=1,\ldots,n)$ be an independent copy of $X$.   
Define $(L^{ij}_s, 0 \le s\le t)$ to be the semimartingale local time process at $0$ of $(X^i-Y^j)/2$, as defined for
example in~\cite[Chapter VI]{ry}.
The total intersection local time is defined by $L=L_t= \sum_{i,j=1}^n L^{ij}_t$.  
\begin{lem}\label{eL}
In the above notation, $\E e^{L}$ is given by \eqref{squares}.
\end{lem}
\begin{proof}
We show, by induction on $k$, that the $k$th term of \eqref{squares} is equal to $\E[L_t^k]/k!$.

First recall that $R^{(n)}_1((t_1,x_1))$ is the sum over $i$ of the  marginal  probability densities for    each $X^i_{t_1}$ evaluated at $x_1$. Consequently $(R^{(n)}_1((t_1,x_1))^2$  can be expressed as the sum over all $i$ and $j$ of  the joint density of $(X^i_{t_1}, Y^j_{t_1})$ evaluated at $(x_1,x_1)$. Integrating over  $x_1$ and $t_1$ gives an expression which agrees with that for $\E[L]$ given by  the occupation time formula ( see \cite[Chapter 6]{ry}).

Similarly  $R^{(n)}_k((t_1,x_1), \ldots, (t_k,x_k))$ is the sum over $i_1, i_2, \ldots,i_k$ of the  densities of the $(X^{i_1}_{t_1}, \ldots, X^{i_k}_{t_k})$ evaluated at $(x_1, x_2, \ldots,x_k)$. Consequently $$(R^{(n)}_k((t_1,x_1), \ldots, (t_k,x_k)))^2$$  can be expressed as the sum over all $i_1, \ldots,i_k$ and $j_1,\ldots j_k$ of  the joint density of $(X^{i_1}_{t_1},\ldots , X^{i_k}_{t_k}, Y^{j_1}_{t_1}, \ldots, Y^{j_k}_{t_k})$ evaluated at $$(x_1,x_2, \ldots, x_k,x_1, x_2,\ldots,x_k).$$ This is integrated over  all $x_i$ and $t_i$ to give an expression for the $k$th term of \eqref{squares}. On the other hand we can derive the same expression for  $\E[L_t^k]$ by writing $L_t^k= k\int_0^t L^{k-1}_s dL_s$, and evaluating the expectation of this using Proposition 3 and Lemma 1 of \cite{fitz}, together with inductive hypothesis.

\end{proof}

Next  will show that, in fact, all exponential
moments of $L_t$ are finite.   First note that $L_t=A+B$, where $A$ is the intersection
local time on the time interval $[0,t/2]$ and $B$ is the remainder.  Thus, by Cauchy-Schwartz, it suffices
to show that $A$ and $B$  each have finite exponential moments of all orders.
Now, on the time interval $[0,t/2]$,  the joint law of the bridges ${\mathbf X}=(X^1,X^2,\ldots X^n)$ and ${\mathbf Y}=(Y^1,Y^2,\ldots,Y^n)$ is equivalent to that of two independent copies of Dyson Brownian motion with Radon-Nikodym density  
a product of two factors each given by \eqref{bridge} with $s=t/2$, $\x=x\1$, $\y=y\1$.  By Lemma~\ref{qn}, this 
Radon-Nikodym density is a bounded random variable.
A similar statement holds on $[t/2,t]$ after time reversal.  It therefore suffices to show that, for two independent Dyson 
Brownian motions the total intersection local time 
has finite exponential moments of all orders.  This is established as a special case of the following lemma which controls the intersection local time for arbitrary starting points of the Dyson Brownian motions.

\begin{prop}\label{expmoments}
Let ${\mathbf X}$ and ${\mathbf Y}$ be independent Dyson Brownian motions starting from points $X_0={\mathbf u}$ and $Y_0={\mathbf v}$  belonging to $\Lambda_n$. Their total intersection time  $L_t$ has finite exponential moments of all orders for all $t>0$. Moreover for any $\beta>0$, and $\epsilon>0$  one may choose $t>0$ small enough that
that 
$$
\E \bigl[ e^{\beta L_t} \bigr ] < 1+\epsilon
$$
uniformly for all ${\mathbf u},{\mathbf v} \in \Lambda_n$.
\end{prop}
\begin{proof}
The processes ${\mathbf X}$ and ${\mathbf Y}$  satisfy a system of SDEs
$$X^i_t= u_i+\beta^i_t+\sum_{j\ne i} \int_0^t \frac{ds}{X^i_s-X^j_s} ,
\qquad Y^i_t=v_i +\gamma^i_t+\sum_{j\ne i} \int_0^t \frac{ds}{Y^i_s-Y^j_s} ,$$
where $\beta^i$, $\gamma^i$, $i=1,2,\ldots,n$ are a collection
of independent standard one-dimensional Brownian motions. Recall this  holds even if $u$ and $v$ lie on the boundary of the Weyl chamber.  By Cauchy-Schwartz, it suffices to
show that for each distinct pair $i,j$, $L_t^{ij}$ has finite exponential moments
of all orders, each of  which can be bounded arbitrarily close to $1$,  uniformly in $u$ and $v$, by choosing $t$ small.

Recall that $L^{ij}$ is the  local time  process of $(X^i-Y^j)/2$ at zero. 
Consequently Tanaka's formula, see ~\cite[Chapter 6]{ry}, states that,
\begin{eqnarray*}
2L^{ij}_t &=& | X^i_t-Y^j_t | -|u_i-v_j|- \int_0^t \sgn(X^i_s-Y^j_s) d(X^i_s-Y^j_s) \\
&=& | X^i_t-Y^j_t|- |u_i-v_j| - \int_0^t \sgn(X^i_s-Y^j_s) d(\beta^i_s-\gamma^j_s) \\
&& - \int_0^t \sgn(X^i_s-Y^j_s) (D^i_s-E^j_s) ds\\
&\le& | X^i_t-u_i|+|Y^j_t -v_j| + \left| \int_0^t \sgn(X^i_s-Y^j_s) d(\beta^i_s-\gamma^j_s) \right| \\ &&
+ \int_0^t |D^i_s| ds +\int_0^t |E^j_s| ds,
\end{eqnarray*}
where
$$D^i_s=\sum_{j\ne i} \frac{1}{X^i_s-X^j_s},\qquad E^j_s=\sum_{j\ne i} \frac{1}{Y^i_s-Y^j_s}.$$
Thus, it suffices to show that each of the random variables
$$| X^i_t-u_i|, \qquad |Y^j_t -v_j| ,\qquad
\left| \int_0^t \sgn(X^i_s-Y^j_s) d(\beta^i_s-\gamma^j_s) \right|,$$
$$\int_0^t |D^i_s| ds \qquad \mbox{and   }\int_0^t |E^j_s| ds,$$
have finite exponential moments of all orders,  each of  which can be bounded arbitrarily close to $1$,  uniformly in $u$ and $v$, by choosing $t$ small. 

 The third of the above random variables is the absolute
value of a Gaussian random variable with mean zero and variance $2t$ and so the desired property holds straightforwardly.  
(To see this, note that the stochastic integral is a continuous martingale with quadratic variation process $2t$, hence is
a Brownian motion by L\'evy's characterisation theorem~\cite[Chapter IV, Theorem 3.6]{ry}.)

To control the exponential moments of the  first and second of the above random variables, we recall that Dyson's Brownian motion arises as the process of eigenvalues of a Brownian motion  $(H_t,t \geq 0)$ in the space of  $n \times n$ Hermitian matrices.  Then $\tilde{H}_t= H_t-H_0$ defines a Hermitian Brownian motion starting from the zero matrix, and applying Weyl's eigenvalue inqualities to the sum $H_0+ \tilde{H}_t$ we deduce that
$$ |X^i_t-u_i| \leq \sigma_t$$
where $\sigma_t$ is the spectral radius of $\tilde{H}_t$. Since  the Brownian motion $(\tilde{H}_t, t \geq 0)$  can be taken to not depend on ${\mathbf u}$, and the  spectral radius of $\tilde{H}_t$ has exponential moments that approach $1$ as $t$ tends $0$, this gives the desired control for $|X^i_t-u_i|$. The same argument applies of course to $|Y^j_t-v_j|$.

The fourth and fifth random variables are essentially the same, so it remains to show that, for each $i$,
$$\xi_i := \int_0^t |D^i_s| ds$$
has finite exponential moments which can be made arbitrarily close to $1$. 
We will prove this by induction over $i$.
For $i<j$, define
$$\xi_{ij}=\int_0^t \frac{1}{X^i_s-X^j_s} ds.$$
First we note that
$$\xi_1= \int_0^t D^i_s ds = X^1_t-u_1-\beta^1_t,$$
has finite exponential moments which may be bounded as desired.  Now, since 
$$\xi_1=\xi_{12}+\cdots+\xi_{1n}$$
and each term is non-negative, this implies that $\xi_{1j}\le\xi_1$  
and hence that $\xi_{1j}$ has exponential moments satisfying the same bound
for each $j=2,\ldots,n$.  Now
$$\xi_2=\xi_{12}+\xi_{23}+\cdots+\xi_{2n}=\int_0^t D^2_s ds+2\xi_{12}=
X^2_t-u_2-\beta^2_t+2\xi_{12}.$$
Thus $\xi_2$ and $\xi_{23},\ldots,\xi_{2n}$ 
all have exponential moments of all orders which may be bounded as desired.  Similarly,
$$\xi_3=\xi_{13}+\xi_{23}+\xi_{34}+\cdots+\xi_{3n}
=\int_0^t D^3_s ds+2\xi_{13}+2\xi_{23}
=X^3_t-u_3-\beta^3_t+2\xi_{13}+2\xi_{23},$$
the fourth term is handled similarly,
and so on.
\end{proof}

We note that, by the stationarity of space-time white noise, the law of $Z_n(t,x,y)/p(t,x,y)^n$
does not depend on $x,y$ and the above proposition implies that $C_t:=E|Z_n(t,x,y)/p(t,x,y)^n|^2<\infty$.

\section{Karlin-McGregor type formula in the white noise setting} \label{KMG}

For $n=1,2,\ldots$ and $\x,\y\in \Lambda_n^\circ$, define
\begin{align}\label{WOE}
K_n(t,\x,\y) = p^*_n(t,\x,\y)\Big( 1+ \sum_{k=1}^\infty \int_{\Delta_k(t)}& \int_{\R^k}
 R^{(\x, \y)}_k((t_1,x_1),\ldots,(t_k,x_k))\\
&\times \W(dt_1,dx_1)\cdots \W(dt_k,dx_k) \Big),\nonumber
\end{align}
where $R^{(\x,\y)}_k$ is the $k$-point correlation function for a collection of $n$
non-intersecting Brownian bridges started at positions $\x=(x_1,x_2, \ldots x_n)$ and ending at positions $\y=(y_1,y_2,\ldots,y_n)$
 at time $t$.
\begin{prop}
The series (\ref{WOE}) is convergent in $L^2(\W)$.  
\end{prop}
\begin{proof}  We need to show that $\E[e^L1_A] <\infty$ where $L$ is the total
intersection local time between two independent sets of $n$ independent Brownian 
bridges started at positions $\x$ and ending at positions $\y$ at time $t$, and $A$ is 
the event that each set is non-intersecting.  So it suffices to show that $\E e^L <\infty$.
By considering pairwise intersection local times and applying H\"older's inequality 
one obtains $\E  e^L \le \E e^{n^2\sqrt{t/2}R}$ where $R$ is the local time at zero
of a standard Brownian bridge on $[0,1]$, which has the Rayleigh distribution
$P(R>r)=e^{-r^2/2}$, $r>0$ (see, for example,~\cite{pitman}).
\end{proof}

In fact, Proposition~\ref{expmoments} yields the following stronger statement.
\begin{prop}\label{K-L2} For each $t>0$, there is a constant $d_t<\infty$ such that for all $\x,\y\in\Lambda_n^\circ$,
$$E |K_n(t,\x,\y)|^2 \le d_t p_n^*(t,\x,\y)\Delta(\x)\Delta(\y).$$
\end{prop}
\begin{proof} As in the proof of Theorem~\ref{con}, let $L$ be the 
 total intersection local time between two independent copies 
of the system of $n$ non-intersecting Brownian bridges
started at positions $\x$ and ending at positions $\y$ at time $t$.
Then (cf. Lemma~\ref{eL})
$$E |K_n(t,\x,\y)|^2=p_n^*(t,\x,\y)^2 Ee^L.$$
Write $L_t=A+B$, where $A$ is the total intersection
local time on the time interval $[0,t/2]$ and $B$ is the remainder. 
By Cauchy-Schwartz,
$$\E e^L\le \big(\E e^{2A}\big)^{1/2} \big(\E e^{2B}\big)^{1/2}.$$
Now, as explained in Section 2,
$$\E e^{2A}=\hat \E\left[ \frac{q_n(t/2,{\bf X}_{t/2},\y)}{q_n(t,\x,\y)} e^{2A} \right],$$
where $\hat E$ denotes the expectation with respect to a Dyson Brownian motion started at $\x$.
By Lemma~\ref{qn} and Proposition~\ref{expmoments} this is bounded by $d_t/q_n(t,\x,\y)$
where $d_t$ is a constant independent of $\x,\y$.  The second term is treated similarly,
and the statement of the proposition follows.
\end{proof}

Before proceeding to the Karlin-McGregor formula, which is the main result of 
this section, we recall the approach of Bertini and Cancrini to the stochastic heat  equation.
 In \cite{bc} these authors make sense of the formal Feynman-Kac representation (\ref{gwf}) by 
 means of smoothing the white noise. For $\kappa>0$ introduce the mollified white noise $W^\kappa$ defined by
$$
W^\kappa(t,x)= \int_0^t \int_{\R} \delta_\kappa(x-y) W(ds,dy),
$$
where $\delta_\kappa(\cdot)$ is the centered Gaussian density of variance $1/\kappa$. 
Then the analogue of (\ref{gwf}) is then meaningful and defines random variables $u^\kappa(t,x,y)$.
 
 Moreover, for each $(t,x,y)$, and any $p\geq 1$,
\begin{equation}\label{conofmol}
u^\kappa(t,x,y) \rightarrow u(t,x,y) \text{ in } L^p(W) \text{ as } \kappa \rightarrow \infty.
\end{equation}

\begin{thm}\label{kmgw}
For $\x,\y\in \Lambda_n^\circ$, $K_n(t,\x,\y) = \det [u(t,x_i,y_j)]_{i,j=1}^n$.
\end{thm}
\begin{proof}
Let $\phi\in E$, multiply both sides by $ \exp^\w \W(\phi)$ and
take expectations.  The lefthand side becomes $K^\phi_n(t,\x,\y)$, defined
earlier by the Feynman-Kac expression (\ref{tildeZphi}). The righthand side becomes 
$$C_n(t,x,y):=\E \left[ \det [u(t,x_i,y_j)]_{i,j=1}^n \exp^\w (\W(\phi)) \right],$$
which will we now argue is also given by (\ref{tildeZphi}), and since $\phi \in E$ is arbitrary 
the statement of the theorem will follow.

Consider the quantity
$$C^\kappa_n(t,\x,\y):=\E \left[ \det [u^\kappa(t,x_i,y_j)]_{i,j=1}^n \exp^\w (\W(\phi)) \right]. $$ Replacing each $u^\kappa(t,x_i,y_j)$ by its Feynman-Kac representation, using Fubini, and integrating over $W$ we obtain
$$ C^\kappa_n(t,\x,\y)=\sum_{\sigma} \sgn(\sigma) p_n(t,\x,\sigma\y)\E \exp \left(  \int^t_0 \psi^\kappa(s, {\mathbf B}^{\sigma}_s)   ds \right)$$
where for each permutation $\sigma$,  ${\mathbf B}^{\sigma}$ is  a  bridge of standard $n$-dimensional Brownian motion starting from $\x$ and ending at $\sigma \y=(y_{\sigma(1)}, \ldots,y_{\sigma(n)})$, and $\psi^\kappa$ is given by 
$$
\psi^\kappa(s, \z)= \sum_{i=1}^n \phi^k(s,z_i) +\sum_{i<j} \delta_{\kappa/2}(z_i-z_j)
$$
with
$$\phi^\kappa(s,z)= \int_{\R} \delta_{\kappa}(z^\prime-z)\phi(s,z^\prime) dz^\prime.
$$
Because $\psi^\kappa$ is invariant under permutations of the coordinates, 
as remarked earlier the argument for the Karlin-McGregor formula of Proposition \ref{kmg} 
allows this to  be rewriten as 
\begin{equation}\label{FKmol}
C^\kappa_n(t,\x,\y)=p^*_n(t,\x,\y) \E\exp\left( \int_0^t \psi^\kappa(s,{\mathbf X}_s) ds\right) ,
\end{equation}
where  ${\mathbf X}=(X^1_s,\ldots,X^n_s, 0\leq s \leq t)$ denotes a collection of non-intersecting 
Brownian bridges which start at positions $\x=(x_1,\ldots,x_n)$ and end at positions $\y=(y_1,\ldots,y_n)$
at time $t$, and $p^*_n(t,\x,\y)$ is the transition density of a Brownian motion in $\Lambda_n$ killed 
when it first hits the boundary.

We now let $\kappa$ tend to infinity. On the one hand, from their definitions and (\ref{conofmol}), we have
\[
C^\kappa_n(t,\x,\y) \rightarrow C_n(t,\x,\y).
\]
On the other hand, since $\phi^\kappa$ converges uniformly to $\phi$, and ${\mathbf X}$ does not visit the boundary of $\Lambda_n$,  the righthand side of equation (\ref{FKmol}) converges to that of (\ref{tildeZphi}).  
This is justified by an application of the Dominated Convergence Theorem using the fact that
$\sup_{x\in \R} l^x_t,
$ where $l^x_t$ denotes the local time of a one-dimensional Brownian motion at level $x$, has finite exponential moments.
\end{proof}

It is known \cite{bc} that for each $x$ the solution to the stochastic equation
$u(t,x,y)$ admits a version that is almost surely continuous in $t$ and
$y$ and moreover is strictly positive. It follows from the above theorem that
for each $\x\in\R^n$, $K_n(s,\x,\z)$ admits a version that is almost surely 
continuous in $t$ and $\z$. Define $K_n(t,\z,\y;s)$ via the chaos expansion (\ref{WOE})
but with the shifted white noise $\dot\W(s+\cdot,\cdot)$.  For each $\y\in\R^n$,
this similarly admits a version that is almost surely continuous in $t$ and $\z$. 
In the following we assume that we are using these versions.  
\begin{cor}\label{K-markov}
For each $\x,\y\in\Lambda_n^\circ$, 
$$K_n(s+t,\x,\y)=\int_{\Lambda_n} K_n(s,\x,\z) K_n(t,\z,\y;s) d\z,$$
almost surely.
Consequently, for each $\x\in\Lambda^\circ_n$, $K_n(t,\x,\cdot),\ t> 0$ is a Markov process taking values in  $C(\Lambda_n^\circ,\R)$.
\end{cor}
\begin{proof} This follows from Theorem~\ref{kmgw} using the (generalised)
Cauchy-Binet formula~\cite{j} together with the corresponding flow property for the solution 
of the stochastic heat equation, namely that for each $x,y\in\R$,
$$u(s+t,x,y)=\int_{\R} u(s,x,z) u(t,z,y;s) dz$$
almost surely.  
\end{proof}

We conclude this section with the following.

\begin{prop}\label{kmg-she}
For each $\x,\y\in\Lambda_n^\circ$, $K_n(t,\x,\y)\ge 0$ almost surely.
\end{prop}
\begin{proof}  In the above notation, we first claim that for each 
$\x,\y\in\Lambda_n^\circ$ and $\kappa>0$, it holds almost surely that
$$\det [u^\kappa(t,x_i,y_j)]_{i,j=1}^n>0.$$
To see this, we use the Feyman-Kac representation, from \cite[(2.17)]{bc},
$$u^\kappa(t,x,y)=p(t,x,y)\E F^\kappa_t(b)$$
where the expectation is with respect to a Brownian bridge $b$
starting at $x$ and ending at $y$ at time $t$, and $F^\kappa_s(b), 0\le s\le t$ 
is an almost surely continuous, strictly positive, multiplicative functional of $b$.
It follows, by a standard path-switching argument 
(see, e.g.,~\cite[Section 1.2]{joc}), that 
\begin{equation}\label{xx}
\det [u^\kappa(t,x_i,y_j)]_{i,j=1}^n=p^*_n(t,\x,\y)\E \prod_{i=1}^n F^\kappa_t(X_i)
\end{equation}
where the expectation is with respect to a collection of $n$ non-intersecting 
Brownian bridges $(X_1,\ldots,X_n)$ started at positions $\x$ and ending at 
positions $\y$ at time $t$.  Indeed, multiplying both sides of \eqref{xx} by $1_A(\y)$, where
$A$ is a measurable subset of $\Lambda_n$, and integrating with respect to
$\y\in\Lambda_n$, gives
\begin{equation}\label{xx1}
\sum_{\sigma\in S_n} (-1)^\sigma \E\big[\prod_{i=1}^n F^\kappa_t(B_i);\ B(t)\in \sigma A\big]
=  \E\big[\prod_{i=1}^n F^\kappa_t(B_i);\ B(t)\in A;\ T>t\big],
\end{equation}
where $B$ is a standard Brownian motion in $\R^n$ started at $\x$ and 
$T$ is the first exit time of $B$ from $\Lambda_n$.  
To prove \eqref{xx} it suffices
to show that \eqref{xx1} holds for every measurable  $A\subset\Lambda_n$.
Let us write
$$\Gamma(B)=\prod_{i=1}^n F^\kappa_t(B_i)$$
and note that \eqref{xx1} is equivalent to
$$\sum_{\sigma\in S_n} (-1)^\sigma \E\big[\Gamma(B);\ B(t)\in \sigma A;\ T\le t\big] =0.$$
Now, 
\begin{align*} 
\sum_{\sigma\in S_n}& (-1)^\sigma \E\big[\Gamma(B);\ B(t)\in \sigma A;\ T\le t\big] \\
&= \sum_{i=1}^{n-1} \sum_{\sigma\in S_n} (-1)^\sigma \E\big[\Gamma(B);\ B(t)\in \sigma A;\ T=T_i\le t\big] 
\end{align*}
where $T_i=\inf\{t\ge 0:\ B_i(t)=B_{i+1}(t)\}$, so it suffices to show that
$$ \sum_{\sigma\in S_n} (-1)^\sigma \E\big[\Gamma(B);\ B(t)\in \sigma A;\ T=T_i\le t\big] =0$$
for each $i=1,\ldots,n-1$.  Fix $i$ and define 
$$\tilde B(t)=\begin{cases} B(t) & t\le T_i\\ s_i B(t) & t\ge T_i\end{cases}$$
where $s_i$ denotes the adjacent transposition $(i,i+1)$.  By the strong Markov property,
$\tilde B$ has the same law as $B$.  Moreover, since $F^\kappa_t$ is a multiplicative 
functional, we also have $\Gamma(\tilde B)=\Gamma(B)$.
Hence,
\begin{align*} \E\big[\Gamma(B);\ B(t)\in \sigma A;\ T=T_i\le t\big] 
&= \E\big[\Gamma(\tilde B);\ \tilde B(t)\in s_i \sigma A;\ T=T_i\le t\big] \\
&= \E\big[\Gamma(B);\ B(t)\in s_i \sigma A;\ T=T_i\le t\big] 
\end{align*}
and it follows that
\begin{align*}
 \sum_{\sigma\in S_n} & (-1)^\sigma \E\big[\Gamma(B);\ B(t)\in \sigma A;\ T=T_i\le t\big] \\
& =  \sum_{\sigma\in S_n} (-1)^\sigma \E\big[\Gamma(B);\ B(t)\in s_i \sigma A;\ T=T_i\le t\big] \\
& = - \sum_{\sigma\in S_n} (-1)^{s_i\sigma} \E\big[\Gamma(B);\ B(t)\in s_i \sigma A;\ T=T_i\le t\big] \\
&= -  \sum_{\sigma\in S_n}  (-1)^\sigma \E\big[\Gamma(B);\ B(t)\in \sigma A;\ T=T_i\le t\big],
\end{align*}
as required.
The result now follows from \eqref{xx}, letting $\kappa\to\infty$.
\end{proof}

\section{On the evolution of the $Z_n$ in the white noise setting}\label{evol}

In this section we discuss the analogue of Theorem \ref{if} in the white
noise setting, and the implication that 
$(\Z_1(t,x,\cdot),\ldots,\Z_n(t,x,\cdot)),\ t\ge 0$ is Markov.

We expect, but will not prove here, that for each $t>0$,
\[
M_n(t,\x,\y)=\frac{K_n(t,\x,\y)}{\Delta(\x)\Delta(\y)}
\]
 has a version which almost surely extends continuously to a strictly
positive function on
$\Lambda_n\times\Lambda_n$.  In particular, for each $t>0$, almost surely,
\begin{equation}\label{cty}
\Z_n(t,a,b) = c_{n,t}  \lim_{\x\to a\1 , \y\to b\1 } M_n(t,\x,\y),
\end{equation}
uniformly on compact intervals.
Assuming this continuity it can be shown
that the analogue of Theorem~\ref{if} holds in the white-noise setting,
that is, if we set $\Z_0=1$ and define, for $n\ge 1$, 
$$ a_n=\frac{n}{t}\frac{\Z_{n-1}\Z_{n+1}}{\Z_{n}^2},$$ 
then, for $t>0$, $x\in\R$ and $\y\in\Lambda_n^\circ$,
\begin{equation}\label{wif}
M_n(t,x\1 ,\y)= \Delta(\y)^{-1} \prod_{i=1}^n u(t,x,y_i) 
\int_{GT(\y)} \prod_{k=1}^{n-1} \prod_{i=1}^{n-k}  a_k(t,x,y^{n-k}_i)
dy^{n-k}_i .
\end{equation}
It is not difficult to see (from the flow property described in Corollary~\ref{K-markov}
of the previous section)
that for each $x\in\R$ and for each $n$, the process 
$$(M_1(t,x\1,\cdot),\ldots,M_n(t,x\1 ,\cdot)),\qquad t\ge 0$$ has the Markov
property. Assuming the validity the formulas (\ref{cty}) and (\ref{wif}) 
this would imply that $(\Z_1(t,x,\cdot),\ldots,\Z_n(t,x,\cdot)),\ t\ge 0$ 
is a Markov process.

A proof of the existence of an almost surely continuous extension for
$M_n(t,\x,\y)$ based on Kolmogorov's criterion would be long and
technical.
Here we will satisfy ourselves with a continuous extension in  $L^2$,
which then allows us to prove (\ref{wif}), and hence the Markov property
of the multi-layer process, in the special case $n=2$.

\begin{lem}\label{A}
For each $t>0$,
\[
M_n(t,\x,\y)=\frac{K_n(t,\x,\y)}{\Delta(\x)\Delta(\y)}
\]
extends continuously in $L^2(\W)$ to $\Lambda_n\times\Lambda_n$.  Moreover
this extension satisfies
\[
\Z_n(t,x,y) = c_{n,t}  M_n(t,x\1 ,y\1 ).
\]
\end{lem}

\begin{proof}

First we recall that we have the representation
\begin{align}
M_n(t,\x,\y) = \frac{p^*_n(t,\x,\y)}{\Delta(\x)\Delta(\y)}
\Big( 1+\sum_{k=1}^\infty \int_{\Delta_k(t)}& \int_{\R^k}
 R^{(\x,\y)}_k((t_1,x_1),\ldots,(t_k,x_k))\\
&\times \W(dt_1,dx_1)\cdots \W(dt_k,dx_k)\Big) ,\nonumber
\end{align}
where $ R^{(\x,\y)}_k$  are the correlations functions of a collection
of $n$ non-intersecting Brownian bridges starting at $\x$ and ending at time
$t$ at $\y$.
Since
\[
\frac{p^*_n(t,\x,\y)}{\Delta(\x)\Delta(\y)}
\]
extends continuously to $\Lambda_n\times\Lambda_n$ this representation
naturally defines the extension of $M_n(t,\x,\y)$ to
$\Lambda_n\times\Lambda_n$. Our task is show continuity in  $L^2(W)$. For
this it is enough to show that
\[
(\x,\y, \x^\prime, \y^\prime) \mapsto \E\bigl[
M_n(t,\x,\y)M_n(t,\x^\prime,\y^\prime) \bigr]
\]
is continuous.  Now, as in the proof of Theorem \ref{con} this expectation is equal to
\[
\frac{p^*_n(t,\x,\y)p^*_n(t,\x^\prime, \y^\prime)}{\Delta(\x)\Delta(\y)\Delta(\x^\prime)\Delta(\y^\prime)}\E[ e^L]
\]
where $L$ is the total intersection local time of two independent sets of
non-intersecting bridges, ${\mathbf X}$ and ${\mathbf X^\prime}$ say, starting at positions $\x=(x_1, \ldots x_n)$ and $\x^\prime=(x^\prime_1,x^\prime_2, \ldots, x^\prime_n)$ and 
ending at $\y=(y_1,\ldots,y_n)$ and $\y^\prime=(y^\prime_1, y^\prime_2, \ldots, y^\prime_n)$ at time $t$.

Let us write $L= L_{[0,\delta]}+L_{[\delta, t-\delta]}+ L_{[t-\delta,t]}$,
where $L_{[0,\epsilon]}$ denotes the local time accrued over the time
periods $[0,\delta]$ and so on.
By conditioning on the position of the bridges at times $\delta$ and
$t-\delta$ we have
\begin{multline*}
\E\bigl[ \exp( L_{[\delta,t-\delta]}) | ({\mathbf X}(0),{\mathbf X}^\prime(0),{\mathbf X}(t),
{\mathbf X}^\prime(t))=(\x,\x^\prime, \y,\y^\prime)\bigr] = \\
\int p((\x,\x^\prime,\y,\y^\prime), ({\boldsymbol \xi},{\boldsymbol \xi}^\prime, {\boldsymbol \eta}, {\boldsymbol \eta}^\prime)) \times \\  \E\bigl[ \exp( L_{[\delta,t-\delta]}) |
(X(\delta),X^\prime(\delta),X(t-\delta), X^\prime(t-\delta))=({\boldsymbol \xi},{\boldsymbol \xi}^\prime, {\boldsymbol \eta}, {\boldsymbol \eta}^\prime)\bigl]
d{\boldsymbol \xi}d{\boldsymbol \xi}^\prime d {\boldsymbol \eta} d {\boldsymbol \eta}^\prime.
\end{multline*}
where the  kernel $p(\cdot,\cdot)$ can be written as a product of transition
densities for non-intersecting Brownian motions, and is thus seen to be
continuous. From this it follows  by a dominated convergence argument that
$$
\E\bigl[ \exp( L_{[\delta,t-\delta]}) | (X(0),X^\prime(0),X(t),
X^\prime(t))=(\x,\x^\prime, \y,\y^\prime)\bigr]$$ depends continuously on $(\x,\x^\prime, \y,\y^\prime)$ also.

To deduce the continuity of 
$$\z \mapsto \E\bigl[ \exp( L) | (X(0),X^\prime(0),X(t), X^\prime(t))=(\x,\x^\prime, \y,\y^\prime)\bigr]$$ 
we must show that the
difference
\begin{align*}
\E\bigl[ \exp( L) | &(X(0),X^\prime(0),X(t), X^\prime(t))=(\x,\x^\prime, \y,\y^\prime)\bigr]\\
&-\E\bigl[
\exp( L_{[\delta,t-\delta]}) | (X(0),X^\prime(0),X(t),
X^\prime(t))=(\x,\x^\prime, \y,\y^\prime)\bigr]
\end{align*}
can be made uniformly small  for $\z$ within  compact sets  by choosing
$\delta$ small enough. Applying the Cauchy-Schwartz inequality this
amounts to showing that
\[
\E\bigl[ \exp( 4 L_{[0,\delta]}) | (X(0),X^\prime(0),X(t),
X^\prime(t))=(\x,\x^\prime, \y,\y^\prime)\bigr] 
\]
and
\[
\E\bigl[ \exp(4L_{[t-\delta,t]}) |
(X(0),X^\prime(0),X(t), X^\prime(t))=(\x,\x^\prime, \y,\y^\prime\bigr]
\]
can be made uniformly close to $1$. This follows from Proposition \ref{expmoments},  noting that the joint  law of ${\mathbf X}$ and ${\mathbf X}^\prime$ over the time interval $[0,\delta]$ is absolutely continuous to that of a pair of independent Dyson Brownian motions, with a density, specified by equation \eqref{bridge} that, by virtue of Lemma \ref{qn}, is bounded uniformly for $(\x,\x^\prime, \y,\y^\prime)$  belonging to compact sets. 
\end{proof}

For each  $t>0$, the continuity in $L^2(W)$ of the mapping $ (\x,\y) \mapsto M_n(t,\x,\y)$ implies the existence of a version of the stochastic  process $M_n(t, \cdot,\cdot)$ which is measurable, see Cohn \cite{cohn}. Henceforth we will always assume that
we are using this version, and likewise with regard to $\Z_n(t, \cdot, \cdot)$.

Recall that $K_n(t,\z,\y;s)$ is defined  via the chaos expansion (\ref{WOE})
but with the shifted white noise $\dot\W(s+\cdot,\cdot)$, and define 
$M_n(t,\z,\y;s)$ from it via
\[
M_n(t,\x,\y;s)=\frac{K_n(t,\x,\y;s)}{\Delta(\x)\Delta(\y)},
\]
\begin{cor}\label{M-markov}
For each $\x,\y\in\Lambda_n$, 
$$M_n(s+t,\x,\y)=\int_{\Lambda_n} M_n(s,\x,\z) M_n(t,\z,\y;s) \Delta(\z)^2 d\z,$$
almost surely.
\end{cor}
\begin{proof}

First note that, for $\x,\y \in \Lambda_n^\circ$ this is an immediate consequence of   the flow property for $K_n$ given by Corollary~\ref{K-markov}. 

For $\y\in \Lambda_n^\circ$, we extend the result to an  $\x\in \Lambda_n \setminus\Lambda_n^\circ$ by taking   a sequence $\x_n$  of points in $\Lambda_n^\circ$ converging to $\x$. Then, by Lemma \ref{A},
\[
M_n(s+t,\x_n,\y) \rightarrow M_n(s+t,\x,\y),
\]
in $L^2(W)$, and hence also in $L^1(W)$. On the otherhand we have,
\begin{multline*}
\E \ \left| \int_{\Lambda_n} M_n(s,\x_n,\z) M_n(t,\z,\y;s) \Delta(\z)^2 d\z- \right. \\
 \left.  \int_{\Lambda_n} M_n(s,\x,\z) M_n(t,\z,\y;s) \Delta(\z)^2 d\z \right|
\\
\leq \int_{\Lambda_n} \E \bigl  [\left| M_n(s,\x_n,\z)- M_n(s,\x,\z)  \right|\bigr ] \, \E\bigl[ M_n(t,\z,\y;s)\bigr] \Delta(\z)^2 d\z,
\end{multline*}
where we have used the independence of$M_n(s,\x,\z)$ and $M_n(t,\z,\y;s)$ and the  positivity of $M_n(t,\z,\y;s)$ which follows from  Proposition~\ref{kmg-she}. Now Lemma \ref{A} certain implies that the integrand  on the righthandside tends to $0$ for  every $z$. Moreover,
\[
\E\bigl[ M_n(s,\x_n,\z)\bigr]=  \frac{p^*_n(s,\x_n,\z)}{\Delta(\x_n)\Delta(\z)}
\]
is uniformly bounded by Lemma \ref{qn}, as is $\E\bigl[ M_n(s,\x,\z)\bigr]$, and
\[
 \int_{\Lambda_n} \E\bigl[ M_n(t,\z,\y;s)\bigr] \Delta(\z)^2 d\z=  \int_{\Lambda_n} p^*_n(t,\z,\y)\frac{\Delta(\z)}{\Delta(\y)}  d\z=1.
\]
Consequently,  by  Dominated Convegence, the integral on the righthandside of the displayed inequalities converges to $0$ and hence the flow property is proved to extend to $\x$.

To further extend the flow property  to $y\in \Lambda_n \setminus\Lambda_n^\circ$, we take $\y_n \in \Lambda_n^\circ$ converging to $\y$, and apply essentially the same arguments again, making use of the result just proved. 
 
\end{proof}

\begin{thm}\label{B}
For $x\in\R$
and $\y\in\Lambda^\circ_2$,
\begin{equation}\label{if2}
\frac{M_2(t,x\1 ,\y)}{u(t,x,y_1)u(t,x,y_2)}=\frac{1}{y_1-y_2}
\int_{y_2}^{y_1} \frac{M_2(t,x\1 ,z\1 )}{u(t,x,z)^2} dz.
\end{equation}
\end{thm}

\begin{proof}
It is known \cite{bc} that the solution to the stochastic equation
$u(t,x,y)$ admits a version that is almost surely continuous in $t$ and
$y$ and moreover is strictly positive. We assume in the following that
we are using this version. In particular, having fixed $t$, $x$ and
$y_1>y_2$ we let $A_\epsilon(x)$ be the event $\{ u(t,x,z) > \epsilon
\text{ for all } z \in [y_2,y_1+1] \}$. Then as $\epsilon \downarrow 0$ we
have ${\mathbb P}(A_\epsilon(x)) \uparrow 1$.

By Theorem~\ref{kmgw}, for $\x,\y\in\Lambda_2^\circ$,
$$M_2(t,\x,\y)=\frac{1}{\Delta(\x)\Delta(\y)}[u(t,x_1,y_1)u(t,x_2,y_2)-u(t,x_1,y_2)u(t,x_2,y_1)].$$
Hence,
\begin{equation}\label{ad}
\frac{M_2(t,\x,\y)}{u(t,x_2,y_1)u(t,x_2,y_2)}
=\frac{1}{\Delta(\x)\Delta(\y)}\left[\frac{u(t,x_1,y_1)}{u(t,x_2,y_1)}-\frac{u(t,x_1,y_2)}{u(t,x_2,y_2)}\right].
\end{equation}
Writing $\y=(z+h,z)$ where $h>0$ this becomes 
\begin{equation}\label{ad1}
\frac{M_2(t,\x,(z+h,z))}{u(t,x_2,z+h)u(t,x_2,z)}
=\frac{1}{(x_1-x_2)h}\left[\frac{u(t,x_1,z+h)}{u(t,x_2,z+h)}-\frac{u(t,x_1,z)}{u(t,x_2,z)}\right].
\end{equation}
Integrating this equation with respect to $z$ over the interval $[y_2,y_1]$ we obtain
\begin{align*}
\int_{y_2}^{y_1}& \frac{M_2(t,\x,(z+h,z))}{u(t,x_2,z+h)u(t,x_2,z)} dz\\
&= \frac{1}{(x_1-x_2) h}\left[ \int_{y_1}^{y_1+h}
\frac{u(t,x_1,z)}{u(t,x_2,z)} dz - \int_{y_2}^{y_2+h}
\frac{u(t,x_1,z)}{u(t,x_2,z)} dz \right].
\end{align*}
Now let $h$ tend to zero. By the continuity of $u(t,x_1, \cdot)$ and
$u(t,x_2, \cdot)$ the RHS converges almost surely to
$$
\frac{1}{(x_1-x_2)} \left[\frac{u(t,x_1,y_1)}{u(t,x_2,y_1)}-\frac{u(t,x_1,y_2)}{u(t,x_2,y_2)}\right] .
$$
We want to identify the limit of the LHS. Consider
\begin{align*}
E&=\left |
\int_{y_2}^{y_1}\frac{M_2(t,\x,(z+h,z))}{u(t,x_2,z+h)u(t,x_2,z)} dz-
\int_{y_2}^{y_1}\frac{M_2(t,\x,z\1 )}{u(t,x_2,z)^2} dz \right| \\
& \leq \int_{y_2}^{y_1}\frac{|M_2(t,\x,(z+h,z))-M_2(t,\x,z\1)|}{u(t,x_2,z)^2} dz\\ &+ \int_{y_2}^{y_1}\frac{|u(t,x_2,z+h)-
u(t,x_2,z)| M_2(t,\x,(z+h,z))}{u(t,x_2,z)^2 u(t,x_2,z+h)} dz .
\end{align*}
We have
\begin{multline}
\E\bigl[ E ; A_\epsilon(x_2) \bigr] \leq  \epsilon^{-2} \int_{y_2}^{y_1}
\E |M_2(t,\x,(z+h,z))-M_2(t,\x,z\1 )| dz  \\
+ \epsilon^{-3} \int_{y_2}^{y_1} \bigl(\E [M_2(t,\x,(z+h,z))^2]
\E[(u(t,x_2,z+h)-u(t,x_2, z))^2] )^{1/2} \bigr) dz
\end{multline}
By virtue of the  uniform continuity in $L_2$ of the mappings $ (z_1,z_2)
\mapsto M_2(t,\x,(z_1,z_2))$ and $ z \mapsto u(t,x_1,z)$  these
integrals tend to zero as $h\downarrow 0$, and consequently $E$ tends to
$0$ in probability.
Thus we have proven
\begin{equation}\label{firststep}
\frac{1}{(x_1-x_2)} \left
[\frac{u(t,x_1,y_1)}{u(t,x_2,y_1)}-\frac{u(t,x_1,y_2)}{u(t,x_2,y_2)}
\right] =\int_{y_2}^{y_1}\frac{M_2(t,\x, z\1  ))}{u(t,x_2,z)^2} dz .
\end{equation}

Next let $\x=(x+h,x)$ and let $h\downarrow 0$.   The LHS of
(\ref{firststep}) can be rewritten as
$$
(y_1-y_2) \frac{M_2(t,(x+h,x),(y_1,y_2))}{u(t,x,y_1)u(t,x,y_2)};
$$
as $h \downarrow 0$ this converges in probability to
$$
(y_1-y_2) \frac{M_2(t,x\1 ,(y_1,y_2))}{u(t,x,y_1)u(t,x,y_2)}.
$$
On the other hand,  if we consider
$$F=\left|\int_{y_2}^{y_1}\frac{M_2(t,(x+h,x), z\1 
))}{u(t,x,z)^2} dz- \int_{y_2}^{y_1}\frac{M_2(t,x\1 , z\1 
))}{u(t,x,z)^2} dz \right|
$$
we have
$$
\E\bigl[ F ; A_\epsilon(x) \bigr]  \leq \epsilon^{-2} \int_{y_2}^{y_1} \E
|M_2(t,(x+h,x), z\1 )-M_2(t,x\1 ,z\1 )| dz
$$
 which again by the $L_2$ continuity of $M_2$ converges to $0$ as $h
\downarrow 0$. From this it follows the RHS of (\ref{firststep})
converges to
$$
\int_{y_2}^{y_1}\frac{M_2(t,x\1 , z\1  ))}{u(t,x,z)^2} dz,
$$
as required.
\end{proof}

We remark that the identity (\ref{firststep}) shows that the ratio of two solutions to the 
stochastic heat equation is in $H^1$; in fact, such ratios have recently been shown 
(in a slightly different setting) by Hairer~\cite{h} to be in $C^{3/2-\epsilon}$.

\begin{cor}\label{markov2} For each $x\in\R$, the process
$$(\Z_1(t,x,\cdot),\Z_2(t,x,\cdot)),\ t\ge 0$$ has the Markov property.
\end{cor}
\begin{proof}
Fix times $0\leq s<t$. Suppose that $F=F(\Z_1(t,x,\cdot),\Z_2(t,x,\cdot))$ is a  bounded  random variable (depending on the random fields at a finite number of points).  We wish to show that the conditional expectation given the white noise $W_{[0,s]}$ of this random variable is measurable with respect to the random fields $( \Z_1(s,x,\cdot),\Z_2(s,x,\cdot)) $.
To see this, note firstly that by Lemma \ref{A}  the same random variable $F$ is a function of the fields  $(M_1(t,x{\bf 1},\cdot),M_2(t,x {\bf 1},\cdot))$. Now the flow property for $M_n$ obtained in  Corollary  \ref{M-markov} together with the independence of $M(t,\cdot,\cdot;s)$ from  $W_{[0,s]}$ implies the conditional expectation $\E \bigl[ F| W_{[0,s]} \bigr]$ has a version $G$ of the form $G= G(M_1(s,x {\bf 1},\cdot),M_2(s,x {\bf 1},\cdot))$. We have $M_1(s,x {\bf 1},\cdot)$ is proportional $\Z_1(t,x,\cdot)$, and more profoundly, by the preceeding Theorem, $M_2(s,x {\bf 1},\cdot))$ can be expressed in terms of  $\Z_1(t,x,\cdot)$ and  $\Z_2(t,x,\cdot)$. Thus we see that $G$ is of the required form.

\end{proof}

It would be interesting to understand the evolution of the multi-layer process 
in terms of a system of stochastic partial differential equations.  
Motivated by the evolution equations obtained in 
Section 2 in the the case of a smooth potential, it is natural to consider and to
try to make sense of the system of equations
\begin{equation}
\partial_t  a_n = \frac12\partial_y^2  a_n + \partial_y[ a_n\partial_y \log \U_n],
\end{equation}
where $\Z_n=\U_1\cdots\U_n$.
For recent progress in this direction, in the case of smooth initial data and periodic boundary
conditions, see~\cite{h}.

\end{document}